\documentclass[a4paper, DIV=12, 11pt]{scrartcl}
\usepackage[latin1]{inputenc}
\usepackage{amsmath}
\usepackage{amsfonts}
\usepackage{amssymb}
\usepackage{amsthm}
\usepackage{graphicx}
\usepackage{thmtools}
\usepackage{bbm}
\usepackage{hyperref}
\usepackage[bottom, marginal]{footmisc}
\usepackage{todonotes}

\usepackage{tikz}
\usetikzlibrary{matrix}

\declaretheoremstyle[headfont=\normalfont]{normalhead}
\newtheorem{lemma}{Lemma}[section]
\newtheorem{theorem}[lemma]{Theorem}
\newtheorem{proposition}[lemma]{Proposition}
\newtheorem{corollary}[lemma]{Corollary}
\newtheorem{definition}[lemma]{Definition}

\newtheorem{maintheorem}{Theorem}

\newcommand{\R}{\mathbb{R}}
\newcommand{\C}{\mathbb{C}}
\newcommand{\U}{\mathrm{U}}
\newcommand{\SO}{\mathrm{SO}}

\DeclareMathOperator{\Val}{Val}
\DeclareMathOperator{\VConv}{VConv}
\DeclareMathOperator{\Conv}{Conv}
\DeclareMathOperator{\vol}{vol}

\DeclareMathOperator{\supp}{supp}

\DeclareMathOperator{\D}{\bar{\mathrm{D}}}

\DeclareMathOperator{\GL}{GL}

\DeclareMathOperator{\GW}{\mathrm{GW}}

\DeclareMathOperator{\Gr}{\mathrm{Gr}}

\DeclareMathOperator{\Kl}{\mathrm{Kl}}

\DeclareMathOperator{\MA}{\mathrm{MA}}

\renewcommand{\Re}{\operatorname{Re}}

\setkomafont{title}{\normalfont\Large}

\author{Jonas Knoerr}

\title{Unitarily invariant valuations on convex functions}
\date{}

\newcommand{\Addresses}{{
		\bigskip
		\footnotesize
		
		Jonas Knoerr, \textsc{Institute of Discrete Mathematics and Geometry, TU Wien, Wiedner Hauptstrasse 8-10, 1040 Wien, Austria}\par\nopagebreak
		\textit{E-mail address}: \texttt{jonas.knoerr@tuwien.ac.at}
		
		\medskip
	}}
	
\makeatletter
\def\blfootnote{\xdef\@thefnmark{}\@footnotetext}
\makeatother

\makeindex
\begin{document}
\maketitle
\begin{abstract}
	Continuous, dually epi-translation invariant valuations on the space of finite-valued convex functions on $\mathbb{C}^n$ that are invariant under the unitary group are investigated. It is shown that elements belonging to the dense subspace of smooth valuations admit a unique integral representation in terms of two families of Monge--Amp\`ere-type operators. In addition, it is proved that homogeneous valuations are uniquely determined by restrictions to subspaces of appropriate dimension and that this information is encoded in the Fourier--Laplace transform of the associated Goodey--Weil distributions. These results are then used to show that a continuous unitarily invariant valuation is uniquely determined by its restriction to a certain finite family of subspaces of $\mathbb{C}^n$.
\end{abstract}

\section{Introduction} 
	One of the most important classical results in integral geometry is Hadwiger's characterization of the space $\Val(\R^n)^{\SO(n)}$ of all continuous, rigid motion invariant valuations on the space $\mathcal{K}(\R^n)$ of convex bodies, that is, all nonempty, convex, compact subsets of $\R^n$ equipped with the Hausdorff metric. Here a map $\mu:\mathcal{K}(\R^n)\rightarrow\R$ is called a valuation if 
	\begin{align*}
		\mu(K\cup L)+\mu(K\cap L)=\mu(K)+\mu(L)
	\end{align*}
	for all $K,L\in\mathcal{K}(\R^n)$ such that $K\cup L\in\mathcal{K}(\R^n)$. Hadwiger showed that $\Val(\R^n)^{\SO(n)}$ is finite dimensional and that a basis of this space is given by the intrinsic volumes \cite{HadwigerVorlesungenuberInhalt1957}. This result directly implies many classical integral geometric formulas, including kinematic formulas as well as Crofton and Cauchy-Kubota formulas.\\
	More recently, Alesker proved that Hadwiger-type theorems hold not only for rigid motion invariant valuations, but also for smaller subgroups of the affine group:
	\begin{theorem}[Alesker \cite{AleskerP.McMullensconjecture2000} Theorem 8.1]
		Let $G\subset \SO(n)$ be a compact subgroup. The space $\Val(\R^n)^{G}$ of all continuous, translation and $G$-invariant valuations on $\mathcal{K}(\R^n)$ is finite dimensional if and only if $G$ operates transitively on the unit sphere.
	\end{theorem}
	His results \cite{AleskerDescriptiontranslationinvariant2001} led to rapid advances in the integral geometry of these transitive compact groups \cite{BernigHadwigertypetheorem2009,BernigIntegralgeometry$G_2$2011,BernigInvariantvaluationsquaternionic2012,BernigSolanesKinematicformulasquaternionic2017,BernigVoideSpininvariantvaluations2016,KotrbatyIntegralgeometryoctonionic2020,WannererIntegralgeometryunitary2014,Wannerermoduleunitarilyinvariant2014}, in particular for complex space forms \cite{BernigEtAlIntegralgeometrycomplex2014}, as well as non-compact groups \cite{AleskerFaifmanConvexvaluationsinvariant2014,BernigFaifmanValuationtheoryindefinite2017,BernigEtAlCurvaturemeasurespseudo2022}.  Alesker \cite{AleskerHardLefschetztheorem2004} was the first to give explicit bases for $\Val(\C^n)^{\U(n)}$. For their study of kinematic formulas for $\Val(\C^n)^{\U(n)}$, Bernig and Fu \cite{BernigFuHermitianintegralgeometry2011} introduced an additional basis, the so-called \emph{Hermitian intrinsic volumes}, which are characterized by their restriction to lower dimensional subspaces in $\C^n$. Before we state their result, recall that a valuation $\mu$ on $\mathcal{K}(\C^n)$ is called $k$-homogeneous if $\mu(tK)=t^k\mu(K)$ for all $t\ge 0$, $K\in\mathcal{K}(\C^n)$. Let us also denote by $E_{k,q}\in \Gr_{k}(\C^n)$ the subspace $E_{k,q}=\C^q\times \R^{k-2q}$, where $\Gr_k(\C^n)$ denotes the space of $k$-dimensional real subspaces of $\C^n$.
	\begin{theorem}[Bernig--Fu \cite{BernigFuHermitianintegralgeometry2011} Theorem 3.2]
		\label{theorem:BernigFuHermitianIntrinsicVolumes}
		For $0\le k\le 2n$, $\max(0,k-n)\le q\le \lfloor\frac{k}{2}\rfloor$ there exists a unique $k$-homogeneous valuation $\mu_{k,q}\in \Val(\C^n)^{\U(n)}$ such that
		\begin{align*}
			\mu_{k,q}|_{E_{k,p}}=\delta_{pq}\vol_k,
		\end{align*}
		where $\vol_k$ denotes the $k$-dimensional Lebesgue measure on $E_{k,p}$. Moreover, these valuations form a basis for $\Val(\C^n)^{\U(n)}$.
	\end{theorem}
	Let us add some more context to this description. By a result due to Hadwiger \cite{HadwigerVorlesungenuberInhalt1957}, the restriction of any $k$-homogeneous, translation invariant, continuous valuation on $\mathcal{K}(\R^n)$ to a $k$-dimensional subspace $E\subset \R^n$ is a multiple of the $k$-dimensional Lebesgue measure, that is, $\mu|_E=\Kl_\mu(E)\cdot \vol_k$ for some $\Kl_\mu(E)\in\R$. The function $\Kl_\mu:\Gr_k(\R^n)\rightarrow\R$ is called the \emph{Klain function} of $\mu$, and due to a result by Klain \cite{KlainshortproofHadwigers1995} it determines $\mu$ uniquely if $\mu$ is an \emph{even }valuation, that is, if it satisfies $\mu(-K)=\mu(K)$ for all $K\in\mathcal{K}(\R^n)$. This applies in particular to any $k$-homogeneous valuation $\mu\in\Val(\C^n)^{\U(n)}$, which is thus uniquely determined by the $\U(n)$-invariant function $\Kl_\mu$ on $\Gr_k(\C^n)$. By associating to every subspace $E\in\Gr_k(\C^n)$ a quantity called the \emph{multiple Kähler angle}, Tasaki \cite{TasakiGeneralizationKahlerangle2001} showed that the orbits of $\U(n)$ on $\Gr_k(\C^n)$ are in $1$-to-$1$ correspondence with the $\lfloor \frac{\min(k,2n-k)}{2}\rfloor$-dimensional simplex. The orbits containing the spaces $E_{k,q}$ correspond exactly to the vertices of this simplex under this identification. Theorem \ref{theorem:BernigFuHermitianIntrinsicVolumes} thus implies that any $k$-homogeneous valuation $\mu\in \Val(\C^n)^{\U(n)}$ is uniquely determined by its restriction to these extremal orbits. \\
	In this article we consider an extension of these results to the functional setting. The notion of valuation generalizes as follows. If $\mathcal{F}$ is a family of (extended) real-valued functions, then we call $\mu:\mathcal{F}\rightarrow\R$ a valuation if
	\begin{align*}
		\mu(f)+\mu(h)=\mu(f\vee h)+\mu(f\wedge h)
	\end{align*}
	whenever the pointwise maximum $f\vee h$ and pointwise minimum $f\wedge h$ belong to $\mathcal{F}$. If $\mathcal{F}$ denotes the family of indicator functions associated to convex bodies, this definition recovers the classical notion of valuations on sets. In recent years, valuations on many different classical spaces have been the focus of intense research, for example Sobolev and $L^p$ spaces \cite{LudwigFisherinformationmatrix2011,LudwigValuationsSobolevspaces2012,LudwigCovariancematricesvaluations2013,MaRealvaluedvaluations2016,Ober$L_p$Minkowskivaluations2014,TsangValuations$Lp$spaces2010,TsangMinkowskivaluations$Lp$2012}, functions of bounded variation \cite{WangSemiValuations$BVRn$2014}, Lipschitz functions \cite{ColesantiEtAlclassinvariantvaluations2020,ColesantiEtAlContinuousvaluationsspace2021} and general Banach lattices  \cite{TradaceteVillanuevaValuationsBanachLattices2020}. Due to their intimate relation to convex bodies, valuations on spaces related to convexity \cite{BobkovEtAlQuermassintegralsquasiconcave2014,CavallinaColesantiMonotonevaluationsspace2015,ColesantiLombardiValuationsspacequasi2017,ColesantiEtAlTranslationinvariantvaluations2018,KoneValuationsOrliczspaces2014} and, more specifically, spaces of convex functions \cite{AleskerValuationsconvexfunctions2019,ColesantiEtAlMinkowskivaluationsconvex2017,ColesantiEtAlHessianvaluations2020,ColesantiEtAlhomogeneousdecompositiontheorem2020,ColesantiEtAlHadwigertheoremconvex2022,ColesantiEtAlHadwigertheoremconvex2024,ColesantiEtAlHadwigertheoremconvex2025,HofstaetterKnoerrEquivariantEndomorphismsConvex2023,Knoerrsupportduallyepi2021,KnoerrSmoothvaluationsconvex2024,MussnigVolumepolarvolume2019} have been one of the most active areas of research in modern valuation theory.\\
	
	Let us introduce some notation. For a finite-dimensional real vector space $V$, let $\Conv(V,\R)$ denote the space of finite-valued convex functions on $V$. We consider this space with the topology of uniform convergence on compact subsets and denote by $\VConv(V)$ the space of all continuous valuations on $\Conv(V,\R)$ that are in addition \emph{dually epi-translation invariant}, that is, that satisfy
	\begin{align*}
		\mu(f+\ell)=\mu(f)\quad\text{for all}~f\in\Conv(V,\R),~\ell:V\rightarrow\R~\text{affine}.
	\end{align*}
	We equip $\VConv(\R^n)$ and all of its subspaces with the topology of uniform convergence on compact subsets in $\Conv(\C^n,\R)$ (see \cite[Proposition 2.4]{Knoerrsupportduallyepi2021} for a description of these subsets). This space is intimately related to continuous translation invariant valuations on convex bodies (see \cite{AleskerValuationsconvexfunctions2019},\cite{Knoerrsupportduallyepi2021}), and this connection has led to the discovery of many structural results for valuations of this type. For example, similar to a classical result by McMullen \cite{McMullenValuationsEulertype1977}, there exists a homogeneous decomposition as shown by Colesanti, Ludwig and Mussnig \cite{ColesantiEtAlhomogeneousdecompositiontheorem2020}: If we denote by $\VConv_k(V)$ the subspace of $\VConv(V)$ consisting of \emph{$k$-homogeneous} valuations, that is, all valuations $\mu\in\VConv(V)$ such that $\mu(tf)=t^k\mu(f)$ for all $t\ge 0$, $f\in\Conv(V,\R)$, then
	\begin{align*}
		\VConv(V)=\bigoplus_{k=0}^{\dim V}\VConv_k(V).
	\end{align*}

	Note that $\VConv_0(V)$ is $1$-dimensional and given by constant valuations. In a series of articles \cite{ColesantiEtAlHadwigertheoremconvex2022,ColesantiEtAlHadwigertheoremconvex2023,ColesantiEtAlHadwigertheoremconvex2024,ColesantiEtAlHadwigertheoremconvex2025} Colesanti, Ludwig and Mussnig also obtained a version of Hadwiger's characterization result for the space $\VConv(\R^n)^{\SO(n)}$ of all rotation invariant valuations in $\VConv(\R^n)$. Their classification result may be summarized as follows: There is precisely one family of invariant valuations for each degree of homogeneity, and any such valuation admits a singular integral representation with respect to the so-called \emph{Hessian measure} of the appropriate degree. Following the analogy with Hadwiger's classification of all continuous and rigid motion invariant valuations on convex bodies, these valuations are called \emph{functional intrinsic volumes}.\\
	Let us remark that the Hessian measures are related to the real Monge--Amp\`ere operator and assign to any convex function a non-negative measure on $\R^n$. As such, they play an important role in the study of so-called Hessian equations, a class of fully non-linear partial differential equations introduced by Trudinger and Wang \cite{TrudingerWangHessianmeasures.I1997,TrudingerWangHessianmeasures.II1999}. For the role of Monge--Amp\`ere-type operators in the construction of valuations on convex functions and convex bodies we refer to \cite{ColesantiEtAlHessianvaluations2020} and \cite{AleskerValuationsconvexsets2005,AleskerValuationsconvexfunctions2019}.\\
	
	This is the first of two articles on the classification of the space $\VConv_k(\C^n)^{\U(n)}$ of all $\U(n)$-invariant valuations in $\VConv_k(\C^n)$, that is, all valuations $\mu\in\VConv_k(\C^n)$ such that
	\begin{align*}
		\mu(f\circ g)=\mu(f)\quad \text{for all }f\in\Conv(\C^n,\R), g\in\U(n).
	\end{align*}
	As $\U(1)=\SO(2)$, the case $n=1$ is already covered by the results by Colesanti, Ludwig and Mussnig, so we will assume that $n\ge2$ without explicitly stating this restriction in the results.\\
	The main idea of our approach is that any valuation in $\VConv_k(\C^n)^{\U(n)}$ should decompose into a sum of valuations that mirror the behavior of the Hermitian intrinsic volumes under restrictions to lower dimensional subspaces. In order to make this precise, let $V$ and $W$ be finite dimensional real vector spaces and consider a linear map $T:V\rightarrow W$. We may then define the pushforward $T_*:\VConv(V)\rightarrow\VConv(W)$ by assigning to $\mu\in \VConv(V)$ the valuation $T_*\mu\in\VConv(W)$ given by
	\begin{align*}
		[T_*\mu](f):=\mu(T^*f)\quad\text{for }f\in\Conv(W,\R).
	\end{align*}
	If $V=\C^n$ and $E\subset \C^n$ is a subspace, we may in particular consider the pushforward along the orthogonal projection $\pi_{E}:\C^n\rightarrow E$, which we call the restriction of a valuation to the subspace $E$. Consider for $\max(0,k-n)\le q\le\lfloor\frac{k}{2}\rfloor$ the closed subspace
	\begin{align*}
		\VConv_{k,q}(\C^n)^{\U(n)}:=\{\mu\in\VConv_k(\C^n)^{\U(n)}:\pi_{E_{k,p}*}\mu=0\text{ for all } p\ne q \}.
	\end{align*}
		Following the analogy with valuations on convex bodies, one may consider valuations in these spaces as functional analogs to the Hermitian intrinsic volumes. Note however that it is not directly clear from the definition whether $\VConv_{k,q}(\C^n)^{\U(n)}$ contains any nontrivial valuations - for example, for $2\le k\le 2n-2$, the functional intrinsic volumes are not contained in any of these spaces.\\
		The overarching goal of this two-part series is to establish that $\VConv_k(\C^n)^{\U(n)}$ decomposes into a direct sum of the spaces $\VConv_{k,q}(\C^n)^{\U(n)}$ and to describe the valuations belonging to these smaller subspaces explicitly in terms of valuations defined in terms of two families of Monge--Amp\`ere-type operators
	\begin{align*}
		&\Theta^n_{k,q}:\Conv(\C^n,\R)\rightarrow\mathcal{M}(\C^n)&&\text{for }\max(0,k-n)\le q\le \left\lfloor\frac{k}{2}\right\rfloor,\text{ and }\\
		&\Upsilon^n_{k,q}:\Conv(\C^n,\R)\rightarrow\mathcal{M}(\C^n)&&\text{for }\max(1,k-n)\le q\le \left\lfloor\frac{k-1}{2}\right\rfloor,
	\end{align*}
	see Definition \ref{definition:ComplexMA}, where $\mathcal{M}(\C^n)$ denotes the space of complex Radon measures on $\C^n$. The family of operators $\Theta^n_{k,q}$ includes the complex Monge--Amp\`ere operator and may be considered as a Hermitian analog of the family of Hessian measures. The operators $\Upsilon^n_{k,q}$ seem to be new, and at the present time, we lack a satisfying geometric interpretation of these functionals. They are obtained from certain invariant differential forms on the cotangent bundle $T^*\C^n$ using a construction from \cite{KnoerrMongeAmpereoperators2024}. We have chosen these operators such that all of the valuations $f\mapsto \int_{\C^n}\phi(|z|)d\Upsilon^n_{k,q}(f;z)$ vanish identically on rotation invariant functions independently of the choice of $\phi\in C_c([0,\infty))$, compare \cite{Knoerrgeometricdecompositionunitarily2024}.\\

	Let us  briefly summarize the key goals of this article and the second part \cite{Knoerrgeometricdecompositionunitarily2024}:\\
	The main contribution of this article is a classification of the dense subspace of smooth valuations in $\VConv_k(\C^n)^{\U(n)}$ in terms of the operators $\Theta^n_{k,q}$ and $\Upsilon^n_{k,q}$ (see Section \ref{section:smoothVal} for the definition of smooth valuations). We will then use a recent characterization result for smooth valuations from \cite{KnoerrPaleyWienerSchwartz2025} to show that an arbitrary continuous valuation in $\VConv_{k}(\C^n)^{\U(n)}$ is uniquely determined by its restriction to the spaces $E_{k,q}$ by approximation.\\
	This implies that the sum of the spaces $\VConv_{k,q}(\C^n)^{\U(n)}$ is direct and dense in $\VConv_{k}(\C^n)^{\U(n)}$, however, it is not enough to establish a direct sum decomposition for the whole space.\\
	In the second part \cite{Knoerrgeometricdecompositionunitarily2024}, we will leverage the explicit description of smooth valuations in order to construct continuous projections onto the different components $\VConv_{k,q}(\C^n)^{\U(n)}$, which provides the desired decomposition of a given continuous valuation in $\VConv_k(\C ^n)^{\U(n)}$. These results are then used to obtain integral representations for all $\U(n)$-invariant valuations in $\VConv_k(\C^n)$ from the description of smooth valuations in terms of the operators $\Theta^n_{k,q}$ and $\Upsilon^n_{k,q}$.
	\subsection{Main results}
	Our first result shows that valuations in $\VConv_k(\C^n)^{\U(n)}$ reflect the behavior of the Hermitian intrinsic volumes under restrictions to the subspaces $E_{k,q}$, $\max(0,k-n)\le q\le\lfloor\frac{k}{2}\rfloor$. 
	\begin{maintheorem}
		\label{maintheorem:vanishing}
			Let $0\le k\le 2n$.	A valuation $\mu\in\VConv_k(\C^n)^{\U(n)}$ satisfies $\mu\equiv 0$ if and only if $\pi_{E_{k,q}*}\mu=0$ for all $\max(0,k-n)\le q\le\lfloor\frac{k}{2}\rfloor$. 
	\end{maintheorem}
	We will obtain Theorem \ref{maintheorem:vanishing} from a characterization of smooth valuations. Denote the subspace of smooth valuations in $\VConv_k(\C^n)$ by $\VConv_k(\C^n)^{sm}$ and consider the spaces
	\begin{align*}
		\VConv_k(\C^n)^{\U(n),sm}&:=\VConv_k(\C^n)^{\U(n)}\cap \VConv_k(\C^n)^{sm}.
	\end{align*}
	Note that $\VConv_k(\C^n)^{\U(n),sm}\subset \VConv_k(\C^n)^{\U(n)}$ is sequentially dense by the main results of \cite{KnoerrSmoothvaluationsconvex2024}. The next result provides  integral representations of these functionals in terms of the Monge--Amp\`ere-type operators $\Theta^n_{k,q}$ and $\Upsilon^n_{k,q}$.
	\begin{maintheorem}
		\label{maintheorem:smooth_unitarily_invariant_valuations}
			Let $0\le k\le 2n$. A valuation $\mu\in\VConv_k(\C^n)^{\U(n)}$ is smooth if and only if there exist $\phi_{q}\in C^\infty_c([0,\infty))$ for $\max(0,k-n)\le q\le\lfloor\frac{k}{2}\rfloor$ and $\psi_q\in C^\infty_c([0,\infty))$ for $\max(1,k-n)\le q\le \lfloor\frac{k-1}{2}\rfloor$ such that
		\begin{align*}
			\mu(f)=\sum_{q=\max(0,k-n)}^{\lfloor\frac{k}{2}\rfloor}\int_{\C^n}\phi_q(|z|^2)d\Theta^n_{k,q}(f;z)+\sum_{q=\max(1,k-n)}^{\lfloor\frac{k-1}{2}\rfloor}\int_{\C^n}\psi_q(|z|^2)d\Upsilon^n_{k,q}(f:z)
		\end{align*}
		for all $f\in\Conv(\C^n,\R)$. For $1\le k\le 2n$, these functions are unique.
	\end{maintheorem}
	In order to obtain Theorem \ref{maintheorem:vanishing} from this characterization, we will explicitly calculate the restriction of these valuations to the spaces $E_{k,q}$. For a smooth function $f\in\Conv(\R^n,\R)$, the measures $\Theta^n_{k,q}(f)$ and $\Upsilon^n_{k,q}(f)$ are given by integrating certain highly nontrivial polynomials in the entries of the Hessian of $f$ over a given Borel set, which makes a direct calculation very cumbersome.\\ 

	To circumvent this problem, we exploit a relation between the restrictions of a homogeneous valuation and the Fourier--Laplace transform of certain distributions associated to these valuations. More precisely, it was shown in \cite{Knoerrsupportduallyepi2021} that one can associate to any $k$-homogeneous valuation $\mu\in\VConv_k(\R^n)$ a compactly supported distribution $\GW(\mu)$ on $(\R^n)^k$, called its Goodey--Weil distribution, and since this distribution has compact support, its Fourier--Laplace transform defines an entire function on $(\C^n)^k$. We will see in Section \ref{section:Fourier} that restricting a $k$-homogeneous valuation to a $k$-dimensional subspace $E\subset\R^n$ corresponds to evaluating this function on the complexified subspace $(E\otimes\C)^k\subset\C^n$. This provides a direct way to connect properties of the restrictions to properties of the original valuation. In particular, we obtain the following.
	\begin{maintheorem}
		\label{maintheorem:simple_valuations}
		If $\mu\in\VConv_k(\R^n)$ satisfies $\pi_{E*}\mu=0$ for all $k$-dimensional subspaces $E\subset \R^n$, then $\mu\equiv 0$.
	\end{maintheorem}
	Let us remark that a version of this result was also obtained by Colesanti, Ludwig and Mussnig in \cite{ColesantiEtAlHadwigertheoremconvex2023} using a very different approach. However, due to the fact that entire functions satisfy strong rigidity conditions, the Fourier--Laplace transform can be used to encode much stronger analytic properties. As a simple application, we will show that a valuation vanishes if it vanishes on convex polynomial functions, see Corollary \ref{corollary:vanishing-on-polynomials}. We also obtain certain restrictions on the support of a valuation, see Section \ref{section:restrictions_support}.\\
	
	As an additional benefit, the Fourier--Laplace transform of Goodey--Weil distributions is rather simple to calculate for valuations with a given integral representation. For smooth $\U(n)$-invariant valuations, this calculation together with the classification in Theorem \ref{maintheorem:smooth_unitarily_invariant_valuations} implies the following decomposition. For $\max(0,k-n)\le q\le \lfloor\frac{k}{2}\rfloor$, set
	\begin{align*}
		\VConv_{k,q}(\C^n)^{\U(n),sm}&:=\VConv_{k,q}(\C^n)^{\U(n)}\cap \VConv_k(\C^n)^{sm}.
	\end{align*}
	\begin{maintheorem}
		\label{maintheorem:DecompSmoothCase}
		We have the following direct sum decomposition for all $0\le k\le 2n$:
		\begin{align*}
			\VConv_k(\C^n)^{\U(n),sm}=\bigoplus_{q=\max(0,k-n)}^{\lfloor\frac{k}{2}\rfloor}\VConv_{k,q}(\C^n)^{\U(n),sm}.
		\end{align*}
	\end{maintheorem}
	Let us briefly comment on the relation of these results to the problem of decomposing $\VConv_k(\C^n)^{\U(n)}$ into a direct sum of the spaces $\VConv_{k,q}(\C^n)^{\U(n)}$. Note that Theorem \ref{maintheorem:vanishing} implies that the sum
	\begin{align*}
		\bigoplus_{q=\max(0,k-n)}^{\lfloor\frac{k}{2}\rfloor} \VConv_{k,q}(\C^n)^{\U(n)}\subset \VConv_k(\C^n)^{\U(n)}
	\end{align*}
	is direct, while Theorem \ref{maintheorem:DecompSmoothCase} shows that it is dense, since the left hand side contains the dense subspace of smooth $\U(n)$-invariant valuations. Moreover, for $k=0,1,2n-1,2n$, we trivially have equality, since there is only one index $\max(0,k-n)\le q\le \lfloor\frac{k}{2}\rfloor$. For all other degrees of homogeneity, we just obtain a direct sum of closed subspaces, and such a sum may of course be dense without providing a decomposition of the whole space. Thus we are left with the problem of decomposing a given valuation in $\VConv_k(\C^n)^{\U(n)}$ into a sum of valuations belonging to the spaces $\VConv_{k,q}(\C^n)^{\U(n)}$, which is the main problem considered in the second part \cite{Knoerrgeometricdecompositionunitarily2024}.
	
	\subsection{Plan of the article}
	In Section \ref{section:Fourier} we introduce and discuss the basic properties of the Fourier--Laplace transform of Goodey--Weil distributions and prove Theorem \ref{maintheorem:simple_valuations}. Section \ref{section:invariantForms} provides a characterization of the differential forms used in the classification of smooth unitarily invariant valuations. In Section \ref{section:characSmoothVal} we calculate the restrictions of valuations defined in terms of the operators $\Theta^n_{k,q}$ and $\Upsilon^n_{k,q}$ and we show that any smooth valuation admits a representation of the type stated in Theorem \ref{maintheorem:smooth_unitarily_invariant_valuations}. The uniqueness of this integral representation is then used to establish Theorem \ref{maintheorem:DecompSmoothCase} and Theorem \ref{maintheorem:vanishing}, which requires an approximation result we discuss in Section \ref{section:smoothVal}.
		
\section{Dually epi-translation invariant valuations}

\subsection{The Fourier--Laplace transform of Goodey--Weil distributions}
\label{section:Fourier}
In this section we recall some facts about Goodey--Weil distributions. Let us remark that this concept is based on ideas by Goodey and Weil \cite{GoodeyWeilDistributionsvaluations1984} in the context of translation invariant valuations on convex bodies. For $\mu\in\VConv_k(V)$, we define its polarization $\bar{\mu}:\Conv(V,\R)^k\rightarrow\R$ by
\begin{align*}
	\bar{\mu}(f_1,\dots,f_k):=\frac{1}{k!}\frac{\partial}{\partial \lambda_1}\Big|_0\dots \frac{\partial}{\partial \lambda_k}\Big|_0\mu\left(\sum_{i=1}^{k}\lambda_if_i\right)\quad \text{for } f_1\dots,f_k\in\Conv(V,\R).
\end{align*}
Essentially, $\bar{\mu}$ is a multilinear functional on $\Conv(V,\R)$, which may be lifted to a distribution.

\begin{theorem}[\cite{Knoerrsupportduallyepi2021} Theorem 2]
	\label{theorem:Def_GW}
	For every $\mu\in\VConv_k(V)$ there exists a unique distribution $\GW(\mu)\in\mathcal{D}'(V^k)$ with compact support which satisfies the following property: If $f_1,...,f_k\in \Conv(V,\R)\cap C^\infty(V)$, then
	\begin{align*}
		\GW(\mu)[f_1\otimes...\otimes f_k]=\bar{\mu}(f_1,...,f_k).
	\end{align*}
\end{theorem}
The distribution $\GW(\mu)$ is called the Goodey--Weil distribution of $\mu\in\VConv_k(V)$ and uniquely determines $\mu$.\\

According to \cite[Theorem 5.5]{Knoerrsupportduallyepi2021}, the support of the Goodey--Weil distribution of a $k$-homogeneous valuation is contained in the diagonal of $V^k$, which leads to the following notion of support: If $\mu=\sum_{k=0}^{\dim V}\mu_k$ is the decomposition of $\mu$ into its homogeneous components, then we set
\begin{align*}
	\supp\mu:=\bigcup_{k=1}^{\dim(V)}\Delta_k^{-1}(\supp\GW(\mu_k))\subset V,
\end{align*}
where $\Delta_k:V\rightarrow V^k$ denotes the diagonal embedding. In particular, $\supp\mu=\emptyset $ for $\mu\in\VConv_0(V)$. Note that the support of $\mu\in\VConv(V)$ is always a compact set. In \cite{Knoerrsupportduallyepi2021} the following alternative characterization of the support was established.
\begin{proposition}[\cite{Knoerrsupportduallyepi2021} Proposition 6.3]
	\label{proposition:characterization_support}
		The support of $\mu\in\VConv(V)$ is minimal (with respect to inclusion) among the closed sets $A\subset V$ with the following property: If $f,g\in\Conv(V,\R)$ satisfy $f=g$ on an open neighborhood of $A$, then $\mu(f)=\mu (g)$.
\end{proposition}

Let us turn to the case $V=\R^n$ with its standard scalar product. As the Goodey--Weil distribution of $\mu\in\VConv_k(\R^n)$ is compactly supported by Theorem \ref{theorem:Def_GW}, its Fourier--Laplace transform is an entire function on $(\C^n)^k\cong (\R^n)^k\otimes \C$ and is given by
\begin{align*}
	\mathcal{F}(\GW(\mu))[w_1,\dots,w_k]=\GW(\mu)[\exp(-i( w_1,\cdot))\otimes\dots\otimes \exp(-i( w_k,\cdot))]
\end{align*}
for $w_1,\dots w_k\in \C^n$. Here, $(\cdot,\cdot):\C^n\times\C^n\rightarrow\C$ is the $\C$ linear extension of the scalar product on $\R^n$. In particular, $\GW(\mu)$ (and thus $\mu$) is uniquely determined by this function. Let us make the following observation.
\begin{corollary}
	\label{corollary:fourier_relation_mu}
	For $\mu\in \VConv_k(\R^n)$, the restriction of $\mathcal{F}(\GW(\mu))$ to $(i\R^n)^k\subset (\C^n)^k$ is given by 
	\begin{align*}
		\mathcal{F}(\GW(\mu))[y_1\otimes i,\dots, y_k\otimes i]=\frac{1}{k!}\frac{\partial^k}{\partial\lambda_1\dots\partial\lambda_k}\Big|_0\mu\left(\sum_{i=1}^k\lambda_i\exp( y_i,\cdot)\right)
	\end{align*}
	for $y_1,\dots,y_k\in\R^n$. Moreover, $\mathcal{F}(\GW(\mu))$ is uniquely determined by this equation.
\end{corollary}
\begin{proof}
	As $\exp( y,\cdot)$ is a convex function for $y\in \R^n$, this follows directly from the characterizing properties of the Goodey--Weil distribution from Theorem \ref{theorem:Def_GW} and the polarization of $\mu$:
		\begin{align*}
			\mathcal{F}(\GW(\mu))[y_1\otimes i,\dots, y_k\otimes i]=&\GW(\mu)[\exp( y_1,\cdot)\otimes\dots\otimes \exp( y_k,\cdot)]\\
			=&\frac{1}{k!}\frac{\partial^k}{\partial\lambda_1\dots\partial\lambda_k}\Big|_0\mu\left(\sum_{i=1}^k\lambda_i\exp( y_i,\cdot)\right).
		\end{align*}
	The last claim follows from the fact that we can recover the power series expansion of the holomorphic function $\mathcal{F}(\GW(\mu))$ in $0$ from the restriction to this subspace of $(\C^n)^k$.
\end{proof}
\begin{proof}[Proof of Theorem \ref{maintheorem:simple_valuations}]
	Assume that $\mu\in\VConv_k(\R^n)$ satisfies $\pi_{W*}\mu=0$ for all orthogonal projections $\pi_W:\R^n\rightarrow W$ onto a $k$-dimensional subspace $W\in \Gr_k(\R^n)$. By Corollary \ref{corollary:fourier_relation_mu}, $\mathcal{F}(\GW(\mu))$ is uniquely determined by
	\begin{align*}
		\mathcal{F}(\GW(\mu))[y_1\otimes i,\dots, y_k\otimes i]
		=&\frac{1}{k!}\frac{\partial^k}{\partial\lambda_1\dots\partial\lambda_k}\Big|_0\mu\left(\sum_{i=1}^{k}\lambda_i \exp(y_i,\cdot)\right)
	\end{align*}
	for $(y_1,\dots,y_k)\in (\R^n)^k$. Set $W:=\mathrm{span}_\R(y_1,\dots,y_k)$ and consider the function $f\in\Conv(W,\R)$ given by
	\begin{align*}
		f:=\sum_{i=1}^{k}\lambda_i \exp(y_i,\cdot).
	\end{align*}
	Then $\pi_W^*f(x)=\sum_{i=1}^{k}\lambda_i\exp(y_i,x)$ for $x\in\R^n$ and thus
	\begin{align*}
		\mu\left(\sum_{i=1}^{k}\lambda_i \exp(y_i,\cdot)\right)=\mu(\pi_W^*f)=0
	\end{align*}
	by assumption. As this holds for all $\lambda_i\ge 0$, we deduce
	\begin{align*}
		\mathcal{F}(\GW(\mu))[y_1\otimes i,\dots, y_k\otimes i]=&\frac{1}{k!}\frac{\partial^k}{\partial\lambda_1\dots\partial\lambda_k}\Big|_0\mu\left(\sum_{i=1}^{k}\lambda_i \exp(y_i,\cdot)\right)=0
	\end{align*}
	for all $(y_1,\dots,y_k)\in(\R^n)^k$ and thus $\mathcal\GW(\mu)=0$. Since $\mu$ is uniquely determined by its Goodey--Weil distribution, this implies $\mu=0$.
\end{proof}
\begin{corollary}
	\label{corollary:vanishing-on-polynomials}
	If $\mu\in\VConv(\R^n)$ vanishes on all convex polynomials, then $\mu=0$.
\end{corollary}
\begin{proof}
	Without loss of generality, we may assume that $\mu$ is $k$-homogeneous, where $0\le k\le n$. For $k=0$ the statement is obviously true, so we may assume $k>0$.\\
	It is easy to see that every polynomial on $\R^n$ may be written as a linear combination of convex polynomials. If $\mu\in\VConv_k(\R^n)$ vanishes on all convex polynomials, this implies that $\GW(\mu)$ vanishes on all functions of the form $P_1\otimes\dots\otimes P_k\in C^\infty((\R^n)^k)$, where $P_1,\dots,P_k$ are polynomials on $\R^n$. Now observe that the sum
	\begin{align*}
		\exp(( y,\cdot))=\sum_{j=0}^{\infty}\frac{( y,\cdot)^j}{j!}
	\end{align*}
	converges locally uniformly in the $C^\infty$-topology for all $y\in \R^n$. In particular,
	\begin{align*}
		\mathcal{F}(\GW(\mu))[y_1\otimes i,\dots,y_k\otimes i]=\sum_{j_1,\dots,j_k=1}^{\infty}\frac{1}{j_1!\dots j_k!}\GW(\mu)[( y_1,\cdot)^{j_1}\otimes\dots\otimes ( y_k,\cdot)^{j_k}]=0
	\end{align*}
	for all $y_1,\dots,y_k\in\R^n$. As in the proof of Theorem \ref{maintheorem:simple_valuations}, this implies $\mu=0$.
\end{proof}

Note that the restriction of $\mu\in\VConv_k(\R^n)$ to a $k$-dimensional subspace defines a valuation of maximal degree on this space. These valuations all admit integral representations with respect to the \emph{real Monge--Amp\`ere operator} $\MA$, which assigns to any $f\in\Conv(\R^n,\R)$ a non-negative Radon measure $\MA(f)$ on $\R^n$, compare \cite{AleskerValuationsconvexfunctions2019}. These measures depend continuously in the weak topology on $f\in\Conv(\R^n,\R)$ and are given by $\det(D^2f(x))dx$ if $f$ is a smooth function, where $D^2f$ denotes the Hessian of $f$. If $E\subset \R^n$ is a $k$-dimensional subspace, we denote the real Monge--Amp\`ere operator on $E$ by $\MA_E$. 
\begin{theorem}[Colesanti--Ludwig--Mussnig \cite{ColesantiEtAlhomogeneousdecompositiontheorem2020} Theorem 5]
	\label{theorem:valuations_top_degree}
	$\mu\in\VConv_n(\R^n)$ if and only if there exists a (necessarily unique) function $\phi\in C_c(\R^n)$ such that
	\begin{align*}
		\mu(f)=\int_{\R^n}\phi d\MA(f).
	\end{align*}
\end{theorem}
To be more precise, the version of this result given by Colesanti, Ludwig, and Mussnig in \cite{ColesantiEtAlhomogeneousdecompositiontheorem2020} only states that any valuation $\mu\in\VConv_n(\R^n)$ admits such an integral representation, however, the uniqueness is essentially contained in their proof. As it is not stated directly and the uniqueness is essential for the arguments in later sections, we provide the missing argument in the next lemma. It provides some information on the general structure of the Fourier--Laplace transform of the Goodey--Weil distributions.
\begin{lemma}
	\label{lemma:interpretation_fourier_maxdegree}
	For $\phi\in C_c(\R^n)$ let $\mu_\phi:=\int_{\R^n}\phi d\MA\in\VConv_n(\R^n)$. Then
	\begin{align*}
		\mathcal{F}(\GW(\mu_\phi))[w_1,\dots,w_n]=\frac{(-1)^n}{n!}\det\left(
		( w_i,w_j)\right)_{i,j=1}^n\mathcal{F}(\phi)\left[\sum_{i=1}^{k}w_i\right].
	\end{align*}
	In particular, $\phi$ is uniquely determined by $\mu_\phi$.\\
	Similarly, let $\mu\in \VConv_k(\R^n)$. For $E\in \Gr_k(\R^n)$ let $\phi_E\in C_c(E)$ denote the unique function such that $\pi_{E*}\mu=\int_E\phi_Ed\MA_E$. Then for $w_1,\dots,w_k\in E_\C:=E\otimes_\R\C$
	\begin{align*}
		\mathcal{F}(\GW(\mu))[w_1,\dots,w_k]=\frac{(-1)^k}{k!}\det\left(
		( w_i,w_j)\right)_{i,j=1}^k\mathcal{F}_E(\phi_E)\left[\sum_{i=1}^{k}w_i\right].
	\end{align*}
	Here $\mathcal{F}_E$ denotes the Fourier--Laplace transform on $L^1(E)$ and $(\cdot,\cdot)$ denotes the $\C$-linear extension of the standard scalar product on $\R^n$.
\end{lemma}
\begin{proof} 
	Note that the second part follows directly from the first as $\mathcal{F}(\GW(\mu))[w_1,\dots,w_k]=\mathcal{F}(\GW(\pi_{E*}\mu))[w_1,\dots,w_k]$ for $w_1,\dots,w_k\in E_\C$ due to Corollary \ref{corollary:fourier_relation_mu}.\\
	Let us show the first equation. As both sides are holomorphic functions, it is sufficient to verify the equation for $w_j=y_j\otimes i$, $1\le j\le n$, where $y_1,\dots,y_n\in\R^n$. Let $D_n(A_1,\dots,A_n)$ denote the mixed discriminant of $n$ symmetric $(n\times n)$-matrices $A_1,\dots,A_n$. Using Corollary \ref{corollary:fourier_relation_mu} we obtain
	\begin{align*}
		\mathcal{F}(\GW(\mu_\phi))[w_1,\dots,w_k]=&\frac{1}{n!}\frac{\partial^n}{\partial\lambda_1\dots\partial\lambda_n}\Big|_0\int_{\R^n}\phi d\MA\left(\sum_{j=1}^{n}\exp( y_j,\cdot)\right)\\
		=&\int_{\R^n}\phi(x) D_n( \exp(y_1,x) y_1\cdot y_1^T,\dots,\exp(y_n,x) y_n\cdot y_n^T )d\vol_n(x)\\
		=&D_n(y_1\cdot y_1^T,\dots, y_n\cdot y_n^T)\int_{\R^n}\phi(x) \exp\left( \sum_{j=1}^ny_n,x\right)d\vol_n(x)\\
		=&\frac{1}{n!}\det\left(\sum_{j=1}^n y_j\cdot y_j^T\right) \mathcal{F}(\phi)\left[\sum_{j=1}^{n}y_j\otimes i\right].
	\end{align*}
	It is easy to see that $\det\left(\sum_{j=1}^n y_j\cdot y_j^T\right)=\det\left(
	( y_i,y_j)\right)_{i,j=1}^n$, which establishes the first formula.\\
	In order to see that $\phi$ is uniquely determined by $\mu_\phi$, note that $\mathcal{F}(\GW(\mu_\phi))$ determines $\mathcal{F}(\phi)$ on a non-empty open set. As $\mathcal{F}(\phi)$ is a holomorphic function, this determines $\mathcal{F}(\phi)$ uniquely and thus also $\phi$, because the Fourier--Laplace transform is injective.
\end{proof} 
Notice that Lemma \ref{lemma:interpretation_fourier_maxdegree} implies that $\mathcal{F}(\GW(\mu))$ is uniquely determined by its values on $k$-tuples of orthogonal vectors. The following corollary shows that the Fourier--Laplace transform of the Goodey--Weil distributions encodes precisely the information obtained from restricting the valuations to subspaces.
\begin{corollary}
	\label{cor:relationFourierRestriction}
	Let $\mu\in\VConv_k(\R^n)$, $E\in\Gr_k(\R^n)$. Then $\mathcal{F}(\GW(\mu))[w_1,\dots,w_k]=0$ for all $w_1,\dots,w_k\in E_\C$ if and only if $\pi_{E*}\mu=0$.
\end{corollary}
\begin{proof}
	Corollary \ref{corollary:fourier_relation_mu} shows that $\pi_{E*}\mu=0$ implies that $\mathcal{F}(\GW(\mu))[w_1,\dots,w_k]=0$ for all $w_1,\dots,w_k\in E$. For the converse statement, let $\phi_E\in C_c(E)$ denote the unique function such that $\pi_{E*}\mu(f)=\int_E\phi_Ed\MA_E(f)$ for all $f\in \Conv(E,\R)$. If $\mathcal{F}(\GW(\mu))$ vanishes on $(E_\C)^k$, then $\mathcal{F}_E(\phi_E)=0$ by  Lemma \ref{lemma:interpretation_fourier_maxdegree}, so $\phi_E=0$, which implies $\pi_{E*}\mu=0$.
\end{proof}

\subsection{Restrictions on the support}
\label{section:restrictions_support}
This section establishes some restrictions on the support of a $k$-homogeneous valuation, which strengthen \cite[Proposition 6.4]{Knoerrsupportduallyepi2021} and are of independent interest. In addition, Corollary \ref{corollary:support-top-degree} is used to establish support restrictions for the integral representation of smooth valuations in Section \ref{section:characSmoothVal}, which are required in \cite{Knoerrgeometricdecompositionunitarily2024}.\\
Before we state these results, let us first add the following fact on the behavior of the support under the pushforward by a linear map $T:V\rightarrow W$ between finite dimensional vector spaces.
\begin{proposition}
	\label{proposition:behavior-support-pushforward}
	Let $T: V\rightarrow W$ be a linear map and $\mu\in\VConv(V)$. Then $\supp(T_*\mu)\subset T(\supp\mu)$.
\end{proposition}
\begin{proof}
	Let $f,g\in \Conv(W,\R)$ be two functions with $f=g$ on a neighborhood $U$ of $T(\supp\mu)$. Then $T^*f$ and $T^*g$ coincide on $T^{-1}(U)$, which is a neighborhood of $\supp\mu$. Thus $[T_*\mu](f)=\mu(T^*f)=\mu(T^*g)=[T_*\mu](g)$ by Proposition \ref{proposition:characterization_support}. As this is true for all $f$ and $g$ with this property, Proposition \ref{proposition:characterization_support} implies $\supp(T_*\mu)\subset T(\supp\mu)$.
\end{proof}
\begin{proposition}
	\label{proposition:shape_support}
	Let $1\le k\le n$ and assume that the support of $\mu\in\VConv_k(V)$ is contained in a $(k-1)$-dimensional affine subspace. Then $\mu=0$. In particular, its support is empty.
\end{proposition}
\begin{proof}
	Let us assume that $V=\R^n$ with its standard scalar product. We will start with valuations of degree $k=n$. If $\mu\in\VConv_n(\R^n)$, there exists a unique function $\phi\in C_c(\R^n)$ such that
	\begin{align*}
		\mu(f)=\int_{\R^n}\phi(x)\det(D^2f(x))dx\quad\text{for all }f\in\Conv(\R^n,\R)\cap C^2(\R^n)
	\end{align*}
	by Theorem \ref{theorem:valuations_top_degree}. Let us assume that the support of $\mu$ is contained in an affine hyperplane $H$ of $\R^n$. Let $H^\pm$ denote the positive and negative open half spaces with respect to some orientation of $H$. If $\psi\in C^\infty_c(\R^n)$ is a function with $\supp\psi\subset\R^n\setminus H$, the characterization of the support in Proposition \ref{proposition:characterization_support} implies
	\begin{align*}
		\mu(f+t\psi)=\mu(f)
	\end{align*}
	for all $t\in\R$ and $f\in\Conv(\R^n,\R)$ such that $f+t\psi$ is convex. We may choose $f(x)=\frac{1}{2}|x|^2$. Then this equation holds for all $t\in(-\epsilon,\epsilon)$ for some $\epsilon>0$. Thus
	\begin{align*}
		0=&\frac{d}{dt}\Big|_0\int_{\R^n}\phi(x)\det(Id_n+tD^2\psi(x))dx=\int_{\R^n}\phi(x)\frac{d}{dt}\Big|_0\det(Id_n+tD^2\psi(x))dx\\
		=&\int_{\R^n}\phi(x)\mathrm{tr}(D^2\psi(x))dx=\int_{\R^n}\phi(x)\Delta\psi(x)dx.
	\end{align*}
	In other words, $\phi$ satisfies $\Delta\phi=0$ on $H^\pm$ in the distributional sense. It is a standard fact from regularity theory that any distributional solution of Laplace's equation is a classical solution. Thus $\phi$ is harmonic on $H^\pm$ and thus in particular analytic on $H^\pm$. However, $\phi$ is compactly supported, so there exists an open set in $H^\pm$ such that $\phi=0$ on this set. By the identity theorem, $\phi= 0$ on the open sets $H^\pm$, that is, $\phi=0$ on the dense open subset $H^+\cup H^-$. Thus $\phi=0$ by continuity, which implies $\mu=0$.\\
	Now let $1\le k\le n-1$ be given and let $\mu\in\VConv_k(\R^n)$ be such that $\supp\mu$ is contained in a $(k-1)$-dimensional affine subspace $H$. If $W$ is a $k$-dimensional subspace, then the image of $H$ under the orthogonal projection $\pi:\R^n\rightarrow W$ is an affine subspace of dimension at most $k-1$. By Proposition \ref{proposition:behavior-support-pushforward}, the support of $\pi_*\mu\in\VConv_k(W)$ is thus contained in an affine subspace of dimension at most $k-1$. However, this is a valuation of degree $k=\dim W$, so the previous argument implies $\pi_*\mu=0$. As this is true for all orthogonal projections $\pi:\R^n\rightarrow W$ onto $k$-dimensional subspaces $W\subset \R^n$, Theorem \ref{maintheorem:simple_valuations} implies $\mu=0$.
\end{proof}
The proof provides the following slightly more refined result for $n$-homogeneous valuations.
\begin{corollary}
	\label{corollary:support-top-degree}
	If $\mu\in\VConv_n(\R^n)$ is given by $\mu=\int_{\R^n}\phi d\MA$ for $\phi\in C_c(\R^n)$, then $\supp\Delta\phi\subset \supp\mu$, where we understand $\Delta\phi$ in the sense of distributions. In particular, $\partial \supp\phi\subset \supp\mu\subset\supp\phi\subset \mathrm{conv}(\supp\mu)$.
\end{corollary}
\begin{proof}
	In the proof of Proposition \ref{proposition:shape_support}, we saw that $\phi$ satisfies $\Delta\phi=0$ on $\R^n\setminus \supp\mu$, so $\supp\Delta\phi\subset \supp\mu$. Recall that $\Delta\phi=0$ on some open set implies that $\phi$ is analytic on this set. As $\phi$ is not analytic in any neighborhood of a given point in the boundary of its support, this implies $\partial \supp\phi\subset \supp \Delta\phi$. Obviously, $\supp\mu\subset \supp\phi$. In order to see that $\supp\phi\subset \mathrm{conv}(\supp\mu)$, note that $\Delta\phi=0$ on the connected unbounded open set $\R^n\setminus \mathrm{conv}(\supp\mu)$. As $\phi$ is compactly supported and analytic on this set, this implies $\phi=0$ on $\R^n\setminus \mathrm{conv}(\supp\mu)$, that is, $\supp\phi\subset \mathrm{conv}(\supp\mu)$.
\end{proof}

Note that the support of a $k$-homogeneous valuation may still be contained in a union of lower dimensional affine subspaces, as shown by the valuation $\mu\in\VConv_1(V)$ with discrete support given by 
\begin{align}
	\label{equation:1-hom_three_points}
	\mu(f)=f(x)+f(-x)-2f(0)\quad\text{for all }f\in\Conv(V,\R),
\end{align}
where $x\in V\setminus\{0\}$ is an arbitrary point. However, the shape of the connected components of the support is still restricted as exemplified by the following result.
\begin{corollary}
	For $k\ge 2$ the support of $\mu\in\VConv_k(V)$ is not discrete unless $\mu=0$.
\end{corollary}
\begin{proof}
	Let $\mu\in\VConv_k(V)$ be a valuation with discrete support, which is thus a finite set due to the compactness of the support. Consider the projection $\pi:V\rightarrow E$ onto a $k$-dimensional subspace. By Proposition \ref{proposition:behavior-support-pushforward}, $\supp\pi_*\mu\subset\pi(\supp\mu)$, so $\supp\pi_*\mu$ is a discrete set. If $\pi_*\mu\ne 0$, then it is a non-trivial valuation of degree $k=\dim E$ on $E$, so its support contains the boundary of an open, non-empty, relatively compact subset by the previous corollary. As $\dim E=k\ge 2$, the boundary of such a subset is not discrete, so we obtain a contradiction. Thus $\pi_*\mu=0$. But this holds for all projections $\pi:V\rightarrow E$ onto $k$-dimensional subspaces of $V$, so Theorem \ref{maintheorem:simple_valuations} implies $\mu=0$.
\end{proof}
Note that the example in \eqref{equation:1-hom_three_points} is $1$-homogeneous and supported on three points. This example is minimal in the following sense:
\begin{corollary}
	If the support of $\mu\in\VConv(V)$ is contained in a two-point set, then $\supp\mu=\emptyset$ and $\mu$ is constant.
\end{corollary}
\begin{proof}
	We may assume that $\mu$ is $k$-homogeneous, $k>0$, $\supp\mu\subset \{x_1,x_2\}$. Then $x_1$ and $x_2$ belong to a $1$-dimensional affine subspace, so for $k>1$ this implies $\mu=0$ by Proposition \ref{proposition:shape_support}. For $k=1$, we may apply Theorem \ref{maintheorem:simple_valuations} and assume that $V=\R$ is $1$-dimensional, $x_1< x_2$, and that $\mu$ is given by $\mu(f)=\int_{\R}\phi(x)f''(x)dx$ for all $f\in\Conv(\R,\R)\cap C^2(\R)$ for some $\phi\in C_c(\R)$, compare Theorem \ref{theorem:valuations_top_degree}. Let $\psi\in C^\infty_c(\R)$ be a function with $\supp\phi\subset\R\setminus\{x_1,x_2\}$.  As in the proof of Proposition \ref{proposition:shape_support}, this implies
	\begin{align*}
		0=\frac{d}{dt}\Big|_0\mu(|\cdot|^2+t\psi)=\int_{\R}\phi(x)\psi''(x)dx.
	\end{align*}
	Thus $\phi$ is a distributional solution of $\phi''=0$ on $(-\infty,x_1)\cup (x_1,x_2)\cup (x_2,\infty)$. As in the proof of Proposition \ref{proposition:shape_support}, this implies that $\phi$ is a classical solution, so $\phi$ is piecewise affine. Since it is compactly supported, this implies $\phi=0$ on $(-\infty,x_1)$ and $(x_2,\infty)$, so $\phi(x_1)=0=\phi(x_2)$ by continuity. But then $\phi=0$ on $(x_1,x_2)$ as well, so $\phi=0$ and thus $\mu=0$.
\end{proof}

\subsection{Smooth valuations on convex functions}
	\label{section:smoothVal}
	
	Smooth valuations on convex functions were originally introduced in \cite{KnoerrSmoothvaluationsconvex2024}, however, we will require a more recent equivalent characterization which allows for more flexible approximation arguments. Recall that we equip $\VConv(\R^n)$ and its subspaces with the topology of uniform convergence on compact subsets. We will use the following definition from \cite{KnoerrPaleyWienerSchwartz2025}:
	\begin{definition}
		\label{definition:smoothValuation}
		A valuation $\mu\in\VConv(\R^n)$ is called smooth if the map
		\begin{align*}
			\R^n &\rightarrow\VConv(\R^n)\\
			x&\mapsto \left[f\mapsto \mu(f(\cdot+x))\right]
		\end{align*}
		is a smooth map between locally convex vector spaces.
	\end{definition}
	We denote the subspace of $\VConv_k(\R^n)$ of smooth valuations by $\VConv_k(\R^n)^{sm}$.\\
	In order to relate this notion to the more geometric characterization from \cite{KnoerrSmoothvaluationsconvex2024}, note that for a smooth convex function, the graph of its differential is an oriented $n$-dimensional submanifold of the cotangent bundle $T^*\R^n\cong \R^n\times \R^n$, and thus defines an integral current. As shown by Fu \cite{FuMongeAmperefunctions.1989}, this construction extends naturally to every convex function $f\in\Conv(\R^n,\R)$, which gives rise to the differential cycle, which we denote by $D(f)$. In \cite{KnoerrSmoothvaluationsconvex2024}, it was shown that for any $\omega\in \Omega^{n-k}_c(\R^n)\otimes \Lambda^{k}(\R^n)^*$, where $\Omega^{n-k}_c(\R^n)$ denotes the space of compactly supported $(n-k)$-forms on $\R^n$ and $\Lambda^{k}(\R^n)^*$ the space of constant differential forms on $\R^n$, the map
	\begin{align*}
		f\mapsto D(f)[\omega]
	\end{align*}
	defines an element in $\VConv_k(\R^n)$. The following characterization was obtained in \cite{KnoerrPaleyWienerSchwartz2025}.
	\begin{theorem}[\cite{KnoerrPaleyWienerSchwartz2025} Theorem D]
		The following are equivalent for $\mu\in\VConv_k(\R^n)$:
		\begin{enumerate}
			\item $\mu$ is a smooth valuation in the sense of Definition \ref{definition:smoothValuation}.
			\item There exists a differential form $\omega\in \Omega^{n-k}_c(\R^n)\otimes \Lambda^{k}(\R^n)^*$ such that
			\begin{align*}
				\mu(f)=D(f)[\omega]
			\end{align*}
			for every $f\in\Conv(\R^n,\R)$.
		\end{enumerate}
	\end{theorem}
	
	We will need the following two approximation results:
	\begin{theorem}[\cite{KnoerrPaleyWienerSchwartz2025} Theorem 3.11]
		\label{theorem:SquentialDensitySmoothVal}
		Let $W\subset\VConv_k(\R^n)$ be a closed and translation invariant subspace. Then smooth valuations are sequentially dense in $W$.
	\end{theorem}
	Moreover, we have the following result for invariant valuations.
		\begin{proposition}[\cite{KnoerrSmoothvaluationsconvex2024} Proposition 6.6]
		\label{proposition:density_invariant_valuations}
		Let $G\subset \GL(n,\R)$ be a compact subgroup. Then the space of smooth $G$-invariant valuations is sequentially dense in the space of all $G$-invariant valuations in $\VConv_k(\R^n)$.
	\end{proposition}
	The following is a special case of \cite[Corollary 6.7]{KnoerrSmoothvaluationsconvex2024}.
	\begin{proposition}
		\label{proposition:representation_by_invariant_form}
		Every smooth valuation $\mu\in\VConv_k(\C^n)^{\U(n)}$  is given by $f\mapsto \mu(f)=D(f)[\omega]$ for a $\U(n)$-invariant differential form $\omega\in \Omega_c^{n-k}(\C^n)\otimes\Lambda^k(\C^n)^*$.
	\end{proposition}
	Note that the converse of this result also holds: If $\omega\in \Omega_c^{n-k}(\C^n)\otimes\Lambda^k(\C^n)^*$ is a $\U(n)$-invariant differential form, then $f\mapsto D(f)[\omega]$ defines a smooth $\U(n)$-invariant valuation in $\VConv_k(\C^n)$. Thus Proposition \ref{proposition:representation_by_invariant_form} reduces the problem of characterizing all smooth $\U(n)$-invariant valuations to a classification of $\U(n)$-invariant differential forms on the cotangent bundle $T^*\C^n\cong \C^n\times\C^n$. However, it turns out that there exist different families of differential forms inducing the same valuations. The relations between these different forms can be obtained from the following description of the kernel of the differential cycle. It uses a certain second order differential operator $\D$, called the \emph{symplectic Rumin operator} in \cite{KnoerrSmoothvaluationsconvex2024}, which was also previously considered in \cite{TsengYauCohomologyHodgetheory2012}. We will not need its precise definition, only that it vanishes on closed forms as well as multiples of the symplectic form $\omega_s$ (see \cite[Proposition 5.13]{KnoerrSmoothvaluationsconvex2024}).
	\begin{theorem}[\cite{KnoerrSmoothvaluationsconvex2024} Theorem 2]
		\label{theorem:Kernel_theorem}
		$\omega\in\Omega_c^{n-k}(\C^n)\otimes\Lambda^k(\C^n)^*$ satisfies $D(f)[\omega]=0$ for all $f\in\Conv(\C^n,\R)$ if and only if
		\begin{enumerate}
			\item $\D\omega=0$,
			\item $\int_{\C^n}\omega=0$, where we consider the zero section $\C^n\hookrightarrow T^*\C^n$ as a submanifold.
		\end{enumerate}
	\end{theorem}
	Note that for $k>0$ the second condition is always satisfied, so a given valuation $\mu\in\VConv_k(\C^n)^{sm}$ is uniquely determined by $\D\omega$ for some representing differential form $\omega$. If $k=0$, then the first condition is always satisfied and we only have to consider the second condition.\\
	As we will see, the relevant space of invariant differential forms is generated (as a module over smooth rotation invariant functions with compact support) by a finite family of invariant forms. This will allow us to use the following reinterpretation of these functionals: Since $D(f)$ is an integral current, we can associate to any $\omega\in \Omega^{n-k}(\C^n)\otimes \Lambda^{k}(\C^n)^*$ the signed measure valued functional 
	\begin{align*}
		\Psi_\omega(f;B):=D(f)[1_{\pi^{-1}(B)}\omega]\quad\text{for}~f\in\Conv(\C^n,\R), B\subset \C^n~\text{bounded Borel set},
	\end{align*}
	where $\pi:\C^n\times\C^n\rightarrow\C^n$ denotes the projection onto the first factor. In other words, the measure $\Psi_\omega(f)$ is characterized by the relation
	\begin{align}
		\label{eq:MAOperators}
		\int_{\C^n}\phi(z)d\Psi_\omega(f;z)=D(f)[\pi^{*}\phi\wedge\omega]\quad\text{for}~\phi\in C_c(\C^n).
	\end{align}
	We refer to \cite{KnoerrMongeAmpereoperators2024} for an interpretation of these functionals as Monge--Amp\`ere-type operators.

\section{Unitarily invariant differential forms}
	\label{section:invariantForms}
		
	In the following sections, we denote the standard Hermitian inner product on $\C^n$ by $\langle\cdot,\cdot\rangle$. In particular, $\langle\cdot,\cdot\rangle$ is $\C$-linear in its first and conjugate $\C$-linear in its second argument.

\subsection{The algebra $[\Lambda^*(\C^n\times\C^n)^*]^{\U(n)}$}
	We start with some results concerning constant invariant forms. Let us equip $T^*\C^n\cong \C^n\times \C^n$ with coordinates $z=x+iy$ on the first factor and induced coordinates $\zeta=\xi+i\eta$ on the second. Consider the following $\U(n)$-invariant differential forms on $\C^n\times\C^n$:
	\begin{align*}
	\theta_0:=&\sum_{j=1}^{n}dx_j\wedge dy_j, &&\theta_1:=\sum_{j=1}^{n}dx_j\wedge d\eta_j-dy_j\wedge d\xi_j,\\
	\theta_2:=&\sum_{j=1}^{n}d\xi_j\wedge d\eta_j, &&	\omega_s:=\sum_{j=1}^{n}dx_j\wedge d\xi_j+dy_j\wedge d\eta_j.
	\end{align*}
	Note that $\omega_s$ is the symplectic form on $\C^n\times\C^n$. Let us consider the space $[\Lambda^*(\C^n\times\C^n)^*]^{\U(n)}$ of all $\U(n)$-invariant forms. Note that this is an algebra with respect to the wedge product.
\begin{theorem}[Park \cite{ParkKinematicformulasreal2002} Theorem 2.12 and 2.13]
	\label{theorem:Parks_theorem}
	The algebra $[\Lambda^*(\C^n\times\C^n)^*]^{\U(n)}$ is generated by $\theta_0,\theta_1,\theta_2,\omega_s$. Moreover, there exists no polynomial relation in degree less or equal to $n$ between these forms.
\end{theorem}	
In other words, the algebra $[\Lambda^*(\C^n\times\C^n)^*]^{\U(n)}$ is isomorphic to a quotient of the polynomial ring $\R[X,Y,Z,W]$ by an ideal that does not contain polynomials of degree less or equal to $n$. This last observation will be key for the discussion below. 
We also need the Lefschetz decomposition (see \cite[Proposition 1.2.30]{HuybrechtsComplexgeometry2005}):
\begin{theorem}
	\label{theorem:Lefschetz-decomposition}
		Let $(W,\omega_s)$ be a symplectic vector space of dimension $2n$ and let $L:\Lambda^* W^*\rightarrow \Lambda^*W^*$, $\tau\mapsto \omega_s\wedge\tau$ be the Lefschetz operator.\\
		For $0\le k\le n$ let $P^kW:=\{\tau\in\Lambda^kW^*:L^{n-k+1}\tau=0\}$ denote the space of primitive $k$-forms on $W$. Then the following holds:
		\begin{enumerate}
			\item There exists a direct sum decomposition $\Lambda^kW^*=\bigoplus_{i\ge 0}L^{i}P^{k-2i}W$.
			\item $L^{n-k}:\Lambda^kW^*\rightarrow\Lambda^{2n-k}W^*$ is an isomorphism.
		\end{enumerate}
\end{theorem}
Note that this decomposition is compatible with linear symplectomorphisms. In particular, if we consider $W=\C^n\times\C^n$ with its natural symplectic form and the diagonal action of $\U(n)$, then $\U(n)$ operates on $W$ by symplectomorphisms. If $\omega\in[\Lambda^k(\C^n\times\C^n)]^{\U(n)}$, this implies that every term in the Lefschetz decomposition of $\omega$ is $\U(n)$-invariant. In particular, every factor of the Lefschetz decomposition of $P(\theta_0,\theta_1,\theta_2,\omega_s)$ for $P\in\R[X,Y,Z,W]$ is again a polynomial in $\theta_0,\theta_1,\theta_2,\omega_s$. We are now able to prove the following key result:
\begin{theorem}
	\label{theorem:divisibility_byW_2}
	Let $P\in \R[X,Y,Z]$ be a homogeneous polynomial of degree $n$. Then there exists a unique homogeneous polynomial $Q\in\R[X,Y,Z,W]$ of degree $n-2$ such that $P(\theta_0,\theta_1,\theta_2)-\omega_s^2\wedge Q(\theta_0,\theta_1,\theta_2,\omega_s)$ is primitive.
\end{theorem}
\begin{proof}
	Let $P\in \R[X,Y,Z]$ be a homogeneous polynomial of degree $n$ and $\tilde{Q}\in\R[X,Y,Z,W]$ the unique homogeneous polynomial of degree $n-1$ given by the Lefschetz decomposition in Theorem \ref{theorem:Lefschetz-decomposition} such that
	\begin{align*}
	\omega_s\wedge[P(\theta_0,\theta_1,\theta_2)-\omega_s\wedge \tilde{Q}(\theta_0,\theta_1,\theta_2,\omega_s)]=0.
	\end{align*}
	We thus have to show that $\tilde{Q}$ is divisible by $W$. Set
	\begin{align*}
	\theta_0^*:=&\theta_0-dx_1\wedge dy_1,&&\tilde{\theta}_0:=dx_1\wedge dy_1,\\
	\theta_1^*:=&\theta_1-(dx_1\wedge d\eta_1-dy_1\wedge d\xi_1), &&\tilde{\theta}_1:=dx_1\wedge d\eta_1-dy_1\wedge d\xi_1,\\
	\theta_2^*:=&\theta_2-d\xi_1\wedge d\eta_1, &&	\tilde{\theta}_2:=d\xi_1\wedge d\eta_1,\\
	\omega_s^*:=&\omega_s-dx_1\wedge d\xi_1+dy_1\wedge d\eta_1, &&\tilde{\omega}_s:=dx_1\wedge d\xi_1+dy_1\wedge d\eta_1.
	\end{align*}
	Note that $\theta^*_0,\theta^*_1,\theta^*_2,\omega_s^*$ may be considered as the generators of $[\Lambda^{*}(\C^{n-1}\times\C^{n-1})]^{\U(n-1)}$. As $\tilde{\theta}_0\wedge\tilde{\theta}_1=\tilde{\theta}_2\wedge\tilde{\theta}_1=\tilde{\theta}_0^2=\tilde{\theta}_2^2=\tilde{\omega}_s\wedge\tilde{\theta_j}=0$ and $\tilde{\theta}_1^3=\tilde{\omega}_s^3=0$, given any polynomial $R\in \R[X,Y,Z,W]$, we have
	\begin{align*}
	&R(\theta_0,\theta_1,\theta_2,\omega_s)\\
	=&R(\theta_0^*,\theta_1^*,\theta_2^*,\omega_s^*)+\frac{\partial R}{\partial W}(\theta_0^*,\theta_1^*,\theta_2^*,\omega_s^*)\wedge\tilde{\omega}_s+\frac{1}{2}\frac{\partial^2 R}{\partial W\partial W}(\theta_0^*,\theta_1^*,\theta_2^*,\omega_s^*)\wedge\tilde{\omega}_s\wedge\tilde{\omega}_s\\
	&+\frac{\partial R}{\partial Y}(\theta_0^*,\theta_1^*,\theta_2^*,\omega_s^*)\wedge\tilde{\theta}_1+\frac{1}{2}\frac{\partial^2 R}{\partial Y\partial Y}(\theta_0^*,\theta_1^*,\theta_2^*,\omega_s^*)\wedge\tilde{\theta}_1\wedge\tilde{\theta}_1\\
	&+\frac{\partial R}{\partial Z}(\theta_0^*,\theta_1^*,\theta_2^*,\omega_s^*)\wedge\tilde{\theta}_2+\frac{\partial R}{\partial X}(\theta_0^*,\theta_1^*,\theta_2^*,\omega_s^*)\wedge \tilde{\theta}_0+\frac{\partial^2 R}{\partial X\partial Z}(\theta_0^*,\theta_1^*,\theta_2^*,\omega_s^*)\wedge \tilde{\theta}_0\wedge\tilde{\theta}_2.
	\end{align*}
	Plugging in the coordinate vector fields $\frac{\partial}{\partial x_1},\frac{\partial}{\partial y_1},\frac{\partial}{\partial \xi_1},\frac{\partial}{\partial \eta_1}$, this implies
	\begin{align*}
	&i_{\frac{\partial}{\partial x_1}}i_{\frac{\partial}{\partial y_1}}i_{\frac{\partial}{\partial \xi_1}}i_{\frac{\partial}{\partial \eta_1}}	R(\theta_0,\theta_1,\theta_2,\omega_s)\\
	&\quad=\left[-\frac{\partial^2 R}{\partial W\partial W}(\theta_0^*,\theta_1^*,\theta_2^*,\omega_s^*)-\frac{\partial^2 R}{\partial Y\partial Y}(\theta_0^*,\theta_1^*,\theta_2^*,\omega_s^*)+\frac{\partial^2 R}{\partial X\partial Z}(\theta_0^*,\theta_1^*,\theta_2^*,\omega_s^*)\right].
	\end{align*}
	We apply this relation to $R(X,Y,Z,W):=W[P(X,Y,Z)-W\tilde{Q}(X,Y,Z,W)]$, which evaluates to zero in $[\Lambda^{*}(\C^{n}\times\C^{n})^*]^{\U(n)}$ by assumption, and obtain
	\begin{align*}
	0=\left[2\tilde{Q}+2W\frac{\partial \tilde{Q}}{\partial W}+W^2\frac{\partial^2\tilde{Q}}{\partial W^2}-W\frac{\partial^2 P}{\partial Y^2}+W^2\frac{\partial^2 \tilde{Q}}{\partial Y^2}+W\frac{\partial ^2 P}{\partial X\partial Z}-W^2\frac{\partial^2 \tilde{Q}}{\partial X\partial Z}\right](\theta_0^*,\theta_1^*,\theta_2^*,\omega_s^*).
	\end{align*}
	But this is a polynomial relation of degree $n-1$ in $[\Lambda^{*}(\C^{n-1}\times\C^{n-1})^*]^{\U(n-1)}$, so  Theorem \ref{theorem:Parks_theorem} implies that this relation holds in $\R[X,Y,Z,W]$. Thus $\tilde{Q}$ is divisible by $W$.
\end{proof}

\subsection{Invariant differential forms}
In addition to the forms considered in the previous section, consider the $1$-forms
\begin{align*}
\gamma_1:=&\sum_{j=1}^{n}x_jdx_j+y_jdy_j, &&\gamma_2:=\sum_{j=1}^{n}x_jdy_j-y_jdx_j,\\
\beta_1:=&\sum_{j=1}^{n}x_jd\xi_j+y_jd\eta_j, &&\beta_2:=\sum_{j=1}^{n}x_jd\eta_j-y_jd\xi_j,
\end{align*}
and set
\begin{align}
	\label{eq:defThetaPrime}
	\begin{split}
		\omega_1:=&\gamma_1\wedge\beta_1+\gamma_2\wedge\beta_2,\\
		\omega_2:=&|z|^2\omega_s-\omega_1,\\
		\theta'_0:=&|z|^2\theta_0-\gamma_1\wedge\gamma_2,\\
		\theta'_1:=&|z|^2\theta_1-(\gamma_1\wedge\beta_2-\gamma_2\wedge\beta_1),\\
		\theta'_2:=&|z|^2\theta_2-\beta_1\wedge\beta_2.
	\end{split}
\end{align}
\begin{proposition}
	\label{proposition:unique_decomposition_invariant_form}
	Let $\omega\in\bigoplus_{k=0}^n(\Omega^{{n-k}}(\C^n)\otimes\Lambda^k(\C^n)^*)^{\U(n)}$. For every $z\in\C^n\setminus\{0\}$ there exist unique polynomials $R,R'\in\R[X,Y,Z,W]$ of degree $n-2$, $R_{\gamma_1\gamma_2},\dots,R_{\beta_1\beta_2}\in\R[X,Y,Z,W]$ of degree $n-1$ such that for all $\zeta\in\C^n$
	\begin{align*}
	\omega|_{(z,\zeta)}=&\gamma_1\wedge\gamma_2\wedge\beta_1\wedge\beta_2\wedge R(\theta'_0,\theta'_1,\theta'_2,\omega_2)+\omega_2^2\wedge R'(\theta'_0,\theta'_1,\theta'_2,\omega_2)\\
	&+\gamma_1\wedge\gamma_2 \wedge R_{\gamma_1\gamma_2}(\theta'_0,\theta'_1,\theta'_2,\omega_2)+\gamma_1\wedge\beta_1\wedge R_{\gamma_1\beta_1}(\theta'_0,\theta'_1,\theta'_2,\omega_2)\\
	&+\gamma_1\wedge\beta_2\wedge R_{\gamma_1\beta_2}(\theta'_0,\theta'_1,\theta'_2,\omega_2)+\gamma_2\wedge\beta_1\wedge R_{\gamma_2\beta_1}(\theta'_0,\theta'_1,\theta'_2,\omega_2)\\
	&+\gamma_2\wedge\beta_2\wedge R_{\gamma_2\beta_2}(\theta'_0,\theta'_1,\theta'_2,\omega_2)+\beta_1\wedge\beta_2 \wedge R_{\beta_1\beta_2}(\theta'_0,\theta'_1,\theta'_2,\omega_2).
	\end{align*}
\end{proposition}
\begin{proof}
	As all of the relevant forms are invariant with respect to translations in the second coordinate of $\C^n\times\C^n$, it is enough to consider the case $\zeta=0$. Using the $\U(n)$-invariance, we may further assume that $z$ is a non-trivial real multiple of $e_1$. Then $\omega|_{(z,0)}$ is invariant under the stabilizer of $e_1$, so
	\begin{align*}
	\omega|_{(z,0)}\in [\Lambda^n(\C^n\times\C^n)^*]^{\U(n-1)},
	\end{align*}
	where $\U(n-1)$ operates diagonally and leaves the first coordinate of $\C^n$ invariant. We may decompose $\C^n=\R z\oplus \R iz\oplus \C^{n-1}$, so setting $U=\R z$, $V=\R iz$, we see that $\omega|_{(z,0)}$ belongs to the space
	\begin{align}
	\label{equation:decomposition_tangent_space}
	\bigoplus\limits_{l_1,k_1,l_2,k_2=0}^1\Lambda^{l_1}U^*\otimes  \Lambda^{k_1}V^*\otimes \Lambda^{l_2}U^*\otimes  \Lambda^{k_2}V^*\otimes [\Lambda^{n-l_1-l_2-k_1-k_2}(\C^{n-1}\times\C^{n-1})^*]^{\U(n-1)}.
	\end{align}
	The space $\Lambda^{l_1}U^*\otimes  \Lambda^{k_1}V^*\otimes \Lambda^{l_2}U^*\otimes \Lambda^{k_2}V^*$ is $1$-dimensional and spanned by a suitable product of the forms $\gamma_1,\gamma_2,\beta_1,\beta_2$, while the space
	$[\Lambda^{n-l_1-l_2-k_1-k_2}(\C^{n-1}\times\C^{n-1})^*]^{\U(n-1)}$ is spanned by homogeneous polynomials of degree $n-l_1-l_2-k_1-k_2$ in $\theta'_0,\theta'_1,\theta'_2,\omega_2$.\\
	Unless $l_1=k_1=l_2=k_2=0$, this gives us the desired polynomial. Moreover this polynomial is unique due to  Theorem \ref{theorem:Parks_theorem}. In the remaining case, we obtain some homogeneous polynomial $\tilde{R}\in\R[X,Y,Z,W]$ of degree $n$ such that the corresponding differential form in the decomposition \eqref{equation:decomposition_tangent_space} is given by $\tilde{R}(\theta'_0,\theta'_1,\theta'_2,\omega_2)$. Note that the polynomial $\tilde{R}$ is in general not be unique, but the differential form $\tilde{R}(\theta'_0,\theta'_1,\theta'_2,\omega_2)$ is. Theorem \ref{theorem:Lefschetz-decomposition} implies that there exists a polynomial $R'\in\R[X,Y,Z,W]$ of degree $n-2$ such that $\tilde{R}(\theta'_0,\theta'_1,\theta'_2,\omega_2)=\omega_2^2\wedge R'(\theta'_0,\theta'_1,\theta'_2,\omega_2)$. Because the square of the Lefschetz operator is injective for the given degree, $R'(\theta'_0,\theta'_1,\theta'_2,\omega_2)$ is uniquely determined by $\tilde{R}(\theta'_0,\theta'_1,\theta'_2,\omega_2)$. As this is a relation of degree $n-2$ in $[\Lambda^{*}(\C^{n-1}\times\C^{n-1})^*]^{\U(n-1)}$, this determines $R'$ uniquely by Theorem \ref{theorem:Parks_theorem}
\end{proof}
We need two simple and well known results about smooth functions, which follow easily from L'Hospital's Theorem.
\begin{lemma}
	\label{lemma:simple_lemma_for_even_differentiable functions}
	If $f:\R\rightarrow\R$ is a smooth even function then there exists $\phi\in C^\infty([0,\infty))$ such that $f(r)=\phi(r^2)$ for all $r\in\R$.
\end{lemma}

\begin{lemma}
	\label{lemma:simple_lemma_vanishing_derivatives}
	If $f\in C^\infty(\R)$ satisfies $f^{(i)}(0)=0$ for all $0\le i\le k$ for some $k\in\mathbb{N}$, then there exists $\phi\in C^\infty(\R)$ such that
	\begin{align*}
	f(r)=r^{k+1}\phi(r)\quad\text{for all }r\in\R.
	\end{align*}
\end{lemma}

Let us denote the subspace of $\R[X,Y,Z,W]$ of $k$-homogeneous polynomials by $\R[X,Y,Z,W]_k$.
\begin{theorem}
	\label{theorem:representation_invariant_differential_form_theta}
	For any $\omega\in\bigoplus_{k=0}^n(\Omega^{{n-k}}(\C^n)\otimes\Lambda^k(\C^n)^*)^{\U(n)}$ there exist \begin{align*}
		&R_\alpha\in C^\infty([0,\infty),\R[X,Y,Z,W]_{n-2}),\\
		&R_{\gamma_1\gamma_2},R_{\gamma_1\beta_1},R_{\gamma_1\beta_2},R_{\gamma_2\beta_1},R_{\gamma_2\beta_2},R_{\beta_1\beta_2}\in C^\infty ([0,\infty),\R[X,Y,Z,W]_{n-1}),\\
		&R_\delta\in C^\infty([0,\infty),\R[X,Y,Z,W]_{n}),
	\end{align*}
	such that for $(z,\zeta)\in \C^n\times\C^n$,
	\begin{align*}
	\omega|_{(z,\zeta)}=&\gamma_1\wedge\gamma_2\wedge\beta_1\wedge\beta_2\wedge R_\alpha(\theta_0,\theta_1,\theta_2,\omega_s)\left[|z|^2\right]\\
	&+\gamma_1\wedge\gamma_2 \wedge R_{\gamma_1\gamma_2}(\theta_0,\theta_1,\theta_2,\omega_s)\left[|z|^2\right]+\gamma_1\wedge\beta_1\wedge R_{\gamma_1\beta_1}(\theta_0,\theta_1,\theta_2,\omega_s)\left[|z|^2\right]\\
	&+\gamma_1\wedge\beta_2\wedge R_{\gamma_1\beta_2}(\theta_0,\theta_1,\theta_2,\omega_s)\left[|z|^2\right]+\gamma_2\wedge\beta_1\wedge R_{\gamma_2\beta_1}(\theta_0,\theta_1,\theta_2,\omega_s)\left[|z|^2\right]\\
	&+\gamma_2\wedge\beta_2\wedge R_{\gamma_2\beta_2}(\theta_0,\theta_1,\theta_2,\omega_s)\left[|z|^2\right]+\beta_1\wedge\beta_2 \wedge R_{\beta_1\beta_2}(\theta_0,\theta_1,\theta_2,\omega_s)\left[|z|^2\right]\\
	&+R_\delta(\theta_0,\theta_1,\theta_2,\omega_s)\left[|z|^2\right].
	\end{align*}
\end{theorem}
\begin{proof}
	Due to the invariance of $\omega$, it is enough to show that such a decomposition holds along the line $\R (e_1,0)$. First, observe that $\omega|_{(0,0)}$ is $\U(n)$-invariant, so there exists a unique polynomial $P\in \R[X,Y,Z,W]$ of degree $n$ such that $\omega|_{(0,0)}=P(\theta_0,\theta_1,\theta_2,\omega_s)$ by Theorem \ref{theorem:Parks_theorem}. By subtracting $P(\theta_0,\theta_1,\theta_2,\omega_s)$ from $\omega$, we may thus assume that $\omega_{(te_1,0)}$ vanishes in $t=0$.\\
	By Proposition \ref{proposition:unique_decomposition_invariant_form}, we may write $\omega|_{(te_1,0)}$ for $t\ne 0$ uniquely as a linear combinations of certain products of $\gamma_1,\gamma_2,\beta_1,\beta_2$ and $R^t(\theta'_0,\theta_1',\theta_2',\omega_2)$ for suitable polynomials $R^t$ in $\R[X,Y,Z,W]$ depending on $t\in\R\setminus{0}$, where $\theta_0',\theta_1',\theta_2'$ are defined by \eqref{eq:defThetaPrime}. But along the line $\R e_1$, 
	\begin{align*}
	&\gamma_1=tdx_1, 
	&&\gamma_2=tdy_1,
	&&&\beta_1=td\xi_1,
	&&&&\beta_2=t d\eta_1,
	\end{align*}
	while for any homogeneous polynomial $R\in\R[X,Y,Z,W]$ of degree $k$
	\begin{align*}
	&R(\theta'_0,\theta'_1,\theta'_2,\omega_2)|_{(te_1,0)}\\
	=&t^{2k}R\left(\sum_{i=2}^{n}dx_i\wedge dy_i,\sum_{i=2}^{n}dx_i\wedge d\eta_i-dy_i\wedge d\xi_i,\sum_{i=2}^{n}d\xi_i\wedge d\eta_i,\sum_{i=2}^{n}dx_i\wedge d\xi_i+dy_i\wedge d\eta_i\right).
	\end{align*}
	We may thus consider $\omega|_{(te_1,0)}$ as a smooth curve in the finite-dimensional vector space spanned by suitable products of the constant forms $dx_1,dy_1,d\xi_1,d\eta_1$ and $\tilde{R}$ for $R\in\R[X,Y,W,W]$ of appropriate degree given by
	\begin{align*}
	\tilde{R}:=R\left(\sum_{i=2}^{n}dx_i\wedge dy_i,\sum_{i=2}^{n}dx_i\wedge d\eta_i-dy_i\wedge d\xi_i,\sum_{i=2}^{n}d\xi_i\wedge d\eta_i,\sum_{i=2}^{n}dx_i\wedge d\xi_i+dy_i\wedge d\eta_i\right).
	\end{align*}
	Note that this corresponds exactly to the evaluation of $R\in\R[X,Y,Z,W]$ in $(\Lambda^*(\C^{n-1}\times\C^{n-1})^*)^{\U(n-1)}$. Let us set $A:=\R[X,Y,Z,W]_{n-2}\oplus(\R[X,Y,Z,W]_{n-1})^6\oplus \R[X,Y,Z,W]_{n-2}$. We will denote the different components of $R\in A$ by $R=(R_\alpha,R_{\gamma_1\gamma_2},\dots, R_{\beta_1\beta_2},R_\delta)$. We obtain a map $	A\rightarrow \Lambda^n(\C^n\times\C^n)^*$ which associates to $(R_\alpha,R_{\gamma_1\gamma_2},\dots, R_{\beta_1\beta_2},R_\delta)\in A$ the differential form
	\begin{align*}
	&dx_1\wedge dy_1\wedge d\xi_1\wedge d\eta_1\wedge \tilde{R}_\alpha +dx_1\wedge dy_1 \wedge \tilde{R}_{\gamma_1\gamma_2}+dx_1\wedge d\xi_1\wedge \tilde{R}_{\gamma_1\beta_1}\\
	&+dx_1\wedge d\eta_1\wedge \tilde{R}_{\gamma_1\beta_2}+dy_1\wedge d\xi_1\wedge \tilde{R}_{\gamma_2\beta_1}+dy_1\wedge d\eta_1\wedge \tilde{R}_{\gamma_2\beta_2}+d\xi_1\wedge d\eta_1 \wedge \tilde{R}
	_{\beta_1\beta_2}+\tilde{W}^2 \tilde{R}_\delta.
	\end{align*}
	Note that this is an isomorphism onto its image: If $R=(R_\alpha,R_{\gamma_1\gamma_2},\dots, R_{\beta_1\beta_2},R_\delta)\in A$ belongs to the kernel of this map, then necessarily $\tilde{R}_\alpha=\tilde{R}_{\gamma_1\gamma_2}=\dots= \tilde{R}_{\beta_1\beta_2}=\tilde{W}^2\tilde{R}_\delta=0$. But these are relations in $[\Lambda^{*}(\C^{n-1}\times\C^{n-1})^*]^{\U(n-1)}$, so Theorem \ref{theorem:Parks_theorem} directly implies $R_\alpha=R_{\gamma_1\gamma_2}=\dots= R_{\beta_1\beta_2}=0$. From the Lefschetz decomposition of $\Lambda^*(\C^{n-1}\times\C^{n-1})^*$, we deduce that $\tilde{W}^2\tilde{R}_\delta=0$ implies $\tilde{R}_\delta=0$, so Theorem \ref{theorem:Parks_theorem} shows $R_\delta=0$. Thus $R=0$, so this is indeed an isomorphism onto its image. Using this isomorphism, we may thus consider $t\mapsto \omega|_{(te_1,0)}$ as a smooth curve in $A$. Let us denote this curve by $R^{\omega}(t)=(R_\alpha^\omega(t),R^\omega_{\gamma_1\gamma_2}(t),\dots, R^\omega_{\beta_1\beta_2}(t),{R}_\delta^\omega(t))$. Using that $\omega$ is $\U(n)$-invariant, this implies that for all $z\in\C^n\setminus\{0\}$, $\omega|_{(z,0)}$ is equal to
	\begin{align*}
	&\gamma_1\wedge\gamma_2\wedge\beta_1\wedge\beta_2\wedge |z|^{-4} R_\alpha^\omega\left(\theta_0,\theta_1,\theta_2,\omega_s\right)+\gamma_1\wedge\gamma_2 \wedge |z|^{-2} R^\omega_{\gamma_1\gamma_2}\left(\theta_0,\theta_1,\theta_2-\frac{\beta_1\wedge\beta_2}{|z|^2},\omega_s\right)\\
	+&\gamma_1\wedge\beta_1\wedge |z|^{-2} R^\omega_{\gamma_1\beta_1}\left(\theta_0,\theta_1,\theta_2,\omega_s-\frac{\gamma_2\wedge\beta_2}{|z|^2}\right)+\gamma_1\wedge\beta_2\wedge |z|^{-2} R^\omega_{\gamma_1\beta_2}\left(\theta_0,\theta_1+\frac{\gamma_2\wedge\beta_1}{|z|^2},\theta_2,\omega_s\right)\\
	+&\gamma_2\wedge\beta_1\wedge |z|^{-2} R^\omega_{\gamma_2\beta_1}\left(\theta_0,\theta_1-\frac{\gamma_1\wedge\beta_2}{|z|^2},\theta_2,\omega_s\right)+\gamma_2\wedge\beta_2\wedge |z|^{-2} R^\omega_{\gamma_2\beta_2}\left(\theta_0,\theta_1,\theta_2,\omega_s-\frac{\gamma_1\wedge\beta_1}{|z|^2}\right)\\
	+&\beta_1\wedge\beta_2 \wedge |z|^{-2} R^\omega_{\beta_1\beta_2}\left(\theta_0-\frac{\gamma_1\wedge\gamma_2}{|z|^2},\theta_1,\theta_2,\omega_s\right)\\
	+&\left(\omega_s-\frac{(\gamma_1\wedge\beta_1+\gamma_2\wedge\beta_2)}{|z|^2}\right)^2\\
	&\wedge R_\delta^\omega\left(\theta_0-\frac{\gamma_1\wedge\gamma_2}{|z|^2},\theta_1-\frac{\gamma_1\wedge\beta_2-\gamma_2\wedge\beta_1}{|z|^2},\theta_2-\frac{\beta_1\wedge\beta_2}{|z|^2},\omega_s-\frac{(\gamma_1\wedge\beta_1+\gamma_2\wedge\beta_2)}{|z|^2}\right),
	\end{align*}
	where we omitted the dependence on $|z|^2$ for the polynomials.	Collecting all terms containing $\gamma_1\wedge\gamma_2\wedge\beta_1\wedge\beta_2$ and setting
	\begin{align*}
	&R'[z^2]
	:=\left(R^\omega_{\alpha}-\frac{\partial R^\omega_{\gamma_1\gamma_2}}{\partial Z}+\frac{\partial R^\omega_{\gamma_1\beta_1}}{\partial W}+\frac{\partial R^\omega_{\gamma_1\beta_2}}{\partial Y}-\frac{\partial R^\omega_{\gamma_2\beta_1}}{\partial Y}-\frac{\partial R^\omega_{\gamma_2\beta_2}}{\partial W}-\frac{\partial R^\omega_{\beta_1\beta_2}}{\partial X}-2\frac{\partial^2(W^2R_\delta^\omega)}{\partial W^2}\right)[z^2],
	\end{align*}
	we see that $\omega|_{(z,0)}$ is given for $z\ne 0$ by
	\begin{align*}
	\omega|_{(z,0)}=&\gamma_1\wedge\gamma_2\wedge\beta_1\wedge\beta_2\wedge|z|^{-4} R'\left(\theta_0,\theta_1,\theta_2,\omega_s\right)\left[|z|^2\right]\\
	&+\gamma_1\wedge\gamma_2 \wedge |z|^{-2}\left[R^\omega_{\gamma_1\gamma{2}}-W\frac{\partial R^\omega_\delta}{\partial X}\right]\left(\theta_0,\theta_1,\theta_2,\omega_s\right)\left[|z|^2\right]\\
	&+\gamma_1\wedge\beta_1\wedge |z|^{-2} \left[R^\omega_{\gamma_1\beta_1}-\frac{\partial (WR^\omega_\delta)}{\partial W}\right]\left(\theta_0,\theta_1,\theta_2,\omega_s\right)\left[|z|^2\right]\\
	&+\gamma_1\wedge\beta_2\wedge |z|^{-2} \left[R^\omega_{\gamma_1\beta_2}-W\frac{\partial R^\omega_\delta}{\partial Y}\right]\left(\theta_0,\theta_1,\theta_2,\omega_s\right)\left[|z|^2\right]\\
	&+\gamma_2\wedge\beta_1\wedge |z|^{-2} \left[R^\omega_{\gamma_2\beta_1}+W\frac{\partial  R^\omega_\delta}{\partial Y}\right]\left(\theta_0,\theta_1,\theta_2,\omega_s\right)\left[|z|^2\right]\\
	&+\gamma_2\wedge\beta_2\wedge |z|^{-2} \left[R^\omega_{\gamma_2\beta_2}-\frac{\partial (W R^\omega_\delta)}{\partial W}\right]\left(\theta_0,\theta_1,\theta_2,\omega_s\right)\left[|z|^2\right]\\
	&+\beta_1\wedge\beta_2 \wedge |z|^{-2} \left[R^\omega_{\beta_1\beta_2}-W\frac{\partial R^\omega_\delta}{\partial Z}\right]\left(\theta_0,\theta_1,\theta_2,\omega_s\right)\left[|z|^2\right]\\
	&+\omega_s^2\wedge R_\delta^\omega\left(\theta_0,\theta_1,\theta_2,\omega_s\right)\left[|z|^2\right],
	\end{align*}
	where the polynomial coefficients are smooth functions that vanish in $t=0$. We may thus apply Lemma \ref{lemma:simple_lemma_for_even_differentiable functions} and Lemma \ref{lemma:simple_lemma_vanishing_derivatives} to see that almost all of these polynomials may be replaced by $\hat{R}(\theta_0,\theta_1,\theta_2,\omega_s)][|z|^2]$ for some $\hat{R}\in C^\infty([0,\infty),\R[X,Y,Z,W])$. This is true for all terms except $|z|^{-4}R'(\theta_0,\theta_1,\theta_2,\omega_s)\left[|z|^2\right]$. Let us set $P(t):=R'[t]-\frac{dR'}{dt}(0)t$. Then the derivatives of $P$ vanish up to order $1$, so we may write $|z|^{-4}P(\theta_0,\theta_1,\theta_2,\omega_s)\left[|z|^2\right]=\hat{P}(\theta_0,\theta_1,\theta_2,\omega_s)\left[|z|^2\right]$ for some $\hat{P}\in C^\infty([0,\infty),\R[X,Y,Z,W])$ according to Lemma \ref{lemma:simple_lemma_for_even_differentiable functions} and Lemma \ref{lemma:simple_lemma_vanishing_derivatives}.\\
	
	Subtracting all of these terms, we may thus assume that $\omega$ is a smooth differential form given for $z\ne 0$ by
	\begin{align*}
	\omega|_{(z,0)}=\gamma_1\wedge\gamma_2\wedge \beta_1\wedge\beta_2\wedge|z|^{-2} P'(\theta_0,\theta_1,\theta_2,\omega_s)
	\end{align*}
	for a polynomial $P'\in\R[X,Y,Z,W]$ homogeneous of degree $n-2$. However, such a form cannot be smooth unless $P'=0$. Consider for example the plane $\R e_1\oplus \R ie_2$. Then for $z:=x_1e_1+y_2ie_2\in\R e_1\oplus \R ie_2$
	\begin{align*}
	&|z|^{-2}\gamma_1\wedge\gamma_2\wedge\beta_1\wedge\beta_2\\
	=&(x_1^2+y_2^2)^{-1}(x_1dx_1+y_2dy_2)\wedge (x_1dy_1-y_2dx_2)\wedge (x_1d\xi_1+y_2d\eta_2)\wedge(x_1d\eta_1-y_2d\xi_2),
	\end{align*}
	so none of the coefficient functions in front of the products of four of the $1$-forms $dx_1,dy_1,d\xi_1,\eta_1$ and $dx_2,dy_2,d\xi_2,\eta_2$ are smooth. Thus the product of $P'(\theta_0,\theta_1,\theta_2,\omega_s)$ with any product of four of these $1$-forms vanishes. This holds for any choice of complex orthonormal coordinates, so in particular $\omega_s^2\wedge P'(\theta_0,\theta_1,\theta_2,\omega_s)=0$. The Lefschetz decomposition implies $P'(\theta_0,\theta_1,\theta_2,\omega_s)=0$.
\end{proof}
\subsection{A pointwise relation}
		For a $1$-form $\eta$ on $T^*\C^n$ let $X_\eta$ denote the unique vector field with $i_{X_\eta}\omega_s=\eta$. For the $1$-forms $\gamma_1,\beta_1$ these vector fields are given by
		\begin{align*}
			&X_{\gamma_1}=\sum_{j=1}^{n}-x_j\frac{\partial}{\partial \xi_j}-y_j\frac{\partial}{\partial \eta_j}, &&X_{\beta_1}=\sum_{j=1}^{n}x_j\frac{\partial}{\partial x_j}+y_j\frac{\partial}{\partial y_j}.		
		\end{align*}
	
	\begin{lemma}
		\label{lemma:decomposition_beta_gamma}
		If $\omega\in\Omega^n(T^*\C^n)$ is primitive, then 
		\begin{align*}
		|z|^2\omega={\gamma_1}\wedge i_{X_{\beta_1}}\omega-{\beta_1}\wedge i_{X_{\gamma_1}}\omega+\omega_s\wedge i_{X_{\gamma_1}}i_{X{\beta_1}}\omega.
		\end{align*}
	\end{lemma}
	\begin{proof}
		Note that $i_{X_{\gamma_1}}{\beta_1}=-|z|^2$, $i_{X_{\beta_1}}{\gamma_1}=|z|^2$. Thus
		\begin{align*}
		i_{X_{\gamma_1}}({\beta_1}\wedge\omega)=-|z|^2\omega-{\beta_1}\wedge i_{X_{\gamma_1}}\omega.
		\end{align*}
		On the other hand, $\omega$ is primitive, that is, $\omega_s\wedge\omega=0$, which implies
		\begin{align*}
		i_{X_{\gamma_1}}({\beta_1}\wedge\omega)=&i_{X_{\gamma_1}}(i_{X_{\beta_1}}\omega_s\wedge\omega)=i_{X_{\gamma_1}}\left(i_{X_{\beta_1}}(\omega_s\wedge\omega)-\omega_s\wedge i_{X{\beta_1}}\omega\right)\\
		=&-i_{X_{\gamma_1}}\omega_s\wedge i_{X_{\beta_1}}\omega-\omega_s\wedge i_{X_{\gamma_1}}i_{X{\beta_1}}\omega=-{\gamma_1}\wedge i_{X_{\beta_1}}\omega-\omega_s\wedge i_{X_{\gamma_1}}i_{X{\beta_1}}\omega.
		\end{align*}
		Thus,
		\begin{align*}
		|z|^2\omega=&-i_{X_{\gamma_1}}({\beta_1}\wedge\omega)-{\beta_1}\wedge i_{X_{\gamma_1}}\omega\\
		=&{\gamma_1}\wedge i_{X_{\beta_1}}\omega+\omega_s\wedge i_{X_{\gamma_1}}i_{X{\beta_1}}\omega-{\beta_1}\wedge i_{X_{\gamma_1}}\omega.
		\end{align*}
	\end{proof}

	\begin{definition}
		For $0\le k\le 2n$, $\max(0,k-n)\le q\le \lfloor \frac{k}{2}\rfloor$ we set
		\begin{align*}
		\theta^n_{k,q}:=\theta_0^{n-k+q}\wedge\theta_1^{k-2q}\wedge\theta_2^q.
		\end{align*}
		To avoid unnecessary distinctions, we set $\theta^n_{k,q}:=0$ if $n,k,q$ do not satisfy these relations.
	\end{definition}
	\begin{corollary}
		\label{corollary:decomposition_theta_tau}
		\begin{align*}
		|z|^2\theta^n_{k,q}\equiv& \gamma_1\wedge[(n-k+q)\gamma_2\wedge\theta^{n-1}_{k,q}+(k-2q)\beta_2\wedge\theta^{n-1}_{k-1,q}]\\
		&+\beta_1\wedge[(k-2q)\gamma_2\wedge\theta^{n-1}_{k-1,q}+q\beta_2\wedge\theta^{n-1}_{k-2,q-1}]\quad\mod\omega_s.
		\end{align*}
	\end{corollary}
	\begin{proof}
		By Theorem \ref{theorem:divisibility_byW_2} there exists a differential form $\xi\in \left(\Lambda^{n-4}(\C^n\times\C^n)^*\right)^{\U(n)}$ such that 
		\begin{align*}
		\omega:=\theta^n_{k,q}-\omega_s^2\wedge\xi
		\end{align*}
		is primitive. We can thus apply Lemma \ref{lemma:decomposition_beta_gamma} to obtain
		\begin{align*}
		|z|^2\theta^n_{k,q}=&|z|^2\omega+|z|^2\omega_s^2\wedge \xi\\
		=&\gamma_1\wedge i_{X_{\beta_1}}\omega-\beta_1\wedge i_{X_{\gamma_1}}\omega+\omega_s\wedge i_{X_{\gamma_1}}i_{X_{\beta_1}}\omega+|z|^2\omega_s^2\wedge \xi\\
		\equiv& \gamma_1\wedge i_{X_{\beta_1}}[\theta^n_{k,q}-\omega_s^2\wedge\xi]-\beta_1\wedge i_{X_{\gamma_1}}[\theta^n_{k,q}-\omega_s^2\wedge\xi]\quad\mod\omega_s\\
		\equiv& \gamma_1\wedge i_{X_{\beta_1}}\theta^n_{k,q}-\beta_1\wedge i_{X_{\gamma_1}}\theta^n_{k,q}\quad\mod\omega_s.
		\end{align*}
		Note that 
		\begin{align*}
		&i_{X_{\gamma_1}}\theta_0=0, &&i_{X_{\gamma_1}}\theta_1=-\gamma_2, &&&i_{X_{\gamma_1}}\theta_2=-\beta_2,\\
		&i_{X_{\beta_1}}\theta_0=\gamma_2, &&i_{X_{\beta_1}}\theta_1=\beta_2, &&&i_{X_{\beta_1}}\theta_2=0.
		\end{align*}
		Thus
		\begin{align*}
		i_{X_{\gamma_1}}\theta^n_{k,q}=&-(k-2q)\gamma_2\wedge\theta_0^{n-k+q}\wedge\theta_1^{k-2q-1}\wedge\theta_2^q-q\beta_2\wedge\theta_0^{n-k+q}\wedge\theta_1^{k-2q}\wedge\theta_2^{q-1}\\
		=&-(k-2q)\gamma_2\wedge\theta^{n-1}_{k-1,q}-q\beta_2\wedge\theta^{n-1}_{k-2,q-1},\\
		i_{X_{\beta_1}}\theta^n_{k,q}=&(n-k+q)\gamma_2\wedge\theta_0^{n-k+q-1}\wedge\theta_1^{k-2q}\wedge\theta_2^q+(k-2q)\beta_2\wedge\theta_0^{n-k+q}\wedge\theta_0^{k-2q-1}\wedge\theta_2^q\\
		=&(n-k+q)\gamma_2\wedge\theta^{n-1}_{k,q}+(k-2q)\beta_2\wedge\theta^{n-1}_{k-1,q},
		\end{align*}
		so we obtain
		\begin{align*}
		|z|^2\theta^n_{k,q}\equiv&\gamma_1\wedge[(n-k+q)\gamma_2\wedge\theta^{n-1}_{k,q}+(k-2q)\beta_2\wedge\theta^{n-1}_{k-1,q}]\\
		&+\beta_1\wedge[(k-2q)\gamma_2\wedge\theta^{n-1}_{k-1,q}+q\beta_2\wedge\theta^{n-1}_{k-2,q-1}]\quad\mod\omega_s.
		\end{align*}	
	\end{proof}

\subsection{Relations for the symplectic Rumin differential}
\label{section:representationSmoothValuations}
In this section we establish some relations between the symplectic Rumin differentials of the invariant forms from the previous section.
\begin{lemma}
	\label{lemma:Rumin_differential_2_1-forms}
	Let $R\in \R[X,Y,Z]$ be a homogeneous polynomial of degree $n-1$, $\phi\in C^\infty_c([0,\infty))$. Define $\psi\in C_c^\infty([0,\infty))$ by $\psi(t):=-\int_t^\infty\phi(s)ds$ for $t\ge 0$.	Then
	\begin{align*}
	\D(\phi(|z|^2) \gamma_1\wedge\gamma_2\wedge R(\theta_0,\theta_1,\theta_2))=&- \D[\psi(|z|^2) \theta_0\wedge R(\theta_0,\theta_1,\theta_2)],\\
	\D(\phi(|z|^2)\gamma_1\wedge\beta_1\wedge R(\theta_0,\theta_1,\theta_2))=&0,\\
	\D(\phi(|z|^2)\gamma_1\wedge\beta_2\wedge R(\theta_0,\theta_1,\theta_2))=&-\frac{1}{2}\D(\psi(|z|^2)\theta_1\wedge R(\theta_0,\theta_1,\theta_2)),\\
	\D(\phi(|z|^2)\gamma_2\wedge\beta_2\wedge R(\theta_0,\theta_1,\theta_2))=&0.
	\end{align*}
	
\end{lemma}
\begin{proof}
	If $\eta$ is any of the forms $\gamma_2,\beta_1,\beta_2$, then
	\begin{align*}
	\D(\phi(|z|^2)\gamma_1\wedge \eta\wedge R(\theta_0,\theta_1,\theta_2))=&\frac{1}{2}\D(d\psi(|z|^2)\eta\wedge R(\theta_0,\theta_1,\theta_2))\\
	=&-\frac{1}{2}\D(\psi(|z|^2)d\eta\wedge R(\theta_0,\theta_1,\theta_2)),
	\end{align*}
	because $\D$ vanishes on closed forms. As $d\gamma_2=2\theta_0$, $d\beta_1=\omega_s$ and $d\beta_2=\theta_1$, this implies the first three relations using that $\D$ vanishes on multiples of $\omega_s$. \\
	For the last equation, let us show that $\D(\phi(|z|^2)\omega_1\wedge R(\theta_0,\theta_1,\theta_2))=0$. First note that $\omega_1\wedge\theta_j'=|z|^2\omega_1\wedge\theta_j$ for all $j=0,1,2$, where $\theta_j'$ denotes the forms defined in \eqref{eq:defThetaPrime}. Let $Q\in\R[X,Y,Z,W]$ be the unique homogeneous polynomial of degree $n-3$ such that $W(R-W^2Q)=0$ in $[\Lambda^*(\C^{n-1}\times\C ^{n-1})^*]^{\mathrm{U}(n-1)}$, which exists due to Theorem \ref{theorem:divisibility_byW_2}. As
	\begin{align*}
	\omega_1\wedge\omega_2^2=&\omega_1\wedge (|z|^2\omega_s-\omega_1)^2=\omega_1\wedge (|z|^4\omega_s^2-2|z|^2\omega_s\wedge\omega_1+\omega_1^2)\\
	=&\omega_1\wedge (|z|^4\omega_s^2-2|z|^2\omega_s\wedge\omega_1)=|z|^2\omega_s\wedge \omega_1\wedge (\omega_2-\omega_1)
	\end{align*}
	is a multiple of $\omega_s$, we obtain in $(0,\zeta)\ne(z,\zeta)\in T^*\C^n$
	\begin{align*}
	&\omega_1\wedge R(\theta_0,\theta_1,\theta_2)\\
	=&|z|^{-2(n-1)}\omega_1\wedge (R-W^2 Q)(|z|^2\theta_0,|z|^2\theta_1,|z|^2\theta_2,\omega_2)\\&+|z|^{-2(n-1)}\omega_1\wedge\omega_2^2\wedge Q(|z|^2\theta_0,|z|^2\theta_1,|z|^2\theta_2,\omega_2)\\
	\equiv& |z|^{-2(n-1)}\omega_1\wedge (R-W^2 Q)(|z|^2\theta_0,|z|^2\theta_1,|z|^2\theta_2,\omega_2)\quad\mod \omega_s\\
	=&|z|^{-2(n-1)}(\omega_1+\omega_2)\wedge (R-W^2 Q)(\theta'_0,\theta'_1,\theta'_2,\omega_2)\\
	=&|z|^{-2(n-1)}|z|^2\omega_s\wedge (R-W^2 Q)(\theta'_0,\theta'_1,\theta'_2,\omega_2).
	\end{align*}
	Here we used that $\omega_2\wedge [R-WQ](\theta'_0,\theta'_1,\theta'_2,\omega_2)=0$ and that $\omega_1\wedge |z|^2\theta_j=\omega_1\wedge\theta'_j$. As $\D$ vanishes on multiples of $\omega_s$, this implies
	\begin{align*}
	\D(\phi(|z|^2)\omega_1\wedge R(\theta_0,\theta_1,\theta_2))|_{(z,\zeta)}=0\quad \text{for }(z,\zeta)\in T^*\C^n, z\ne0,
	\end{align*}
	so $\D(\phi(|z|^2)\omega_1\wedge R(\theta_0,\theta_1,\theta_2))=0$ by continuity. In particular,
	\begin{align*}
	&\D(\phi(|z|^2)\gamma_2\wedge\beta_2\wedge R(\theta_0,\theta_1,\theta_2))\\
	=&\D(\phi(|z|^2)\gamma_2\wedge\beta_2\wedge R(\theta_0,\theta_1,\theta_2))+\D(\phi(|z|^2)\gamma_1\wedge\beta_1\wedge R(\theta_0,\theta_1,\theta_2))\\
	=&\D(\phi(|z|^2)\omega_1\wedge R(\theta_0,\theta_1,\theta_2))=0.
	\end{align*}
\end{proof}

\begin{lemma}
	\label{lemma:Rumin_differential_pointwise_relation}
	For $\phi\in C_c^\infty([0,\infty))$ define $\psi\in C_c^\infty([0,\infty))$ by $\psi(t):=-\int_t^\infty\phi(s)ds$ for $t\ge 0$. Then
	\begin{align*}
	&\D\left(\left[\phi(|z|^2)|z|^2+\frac{2n-k}{2}\psi(|z|^2)\right]\theta^n_{k,q}\right)\\
	=&(k-2q)\D\left(\phi(|z|^2)\beta_1\wedge\gamma_2\wedge\theta^{n-1}_{k-1,q}\right)+q\D\left(\phi(|z|^2)\beta_1\wedge\beta_2\wedge\theta^{n-1}_{k-2,q-1}\right).
	\end{align*}
\end{lemma}
\begin{proof}
	As $\psi'=\phi$, Lemma \ref{lemma:Rumin_differential_2_1-forms} implies
	\begin{align*}
	\D(\psi(|z|^2)\theta^n_{k,q})=-\D(\phi(|z|^2)\gamma_1\wedge\gamma_2\wedge\theta^{n-1}_{k,q})=-2\D(\phi(|z|^2)\gamma_1\wedge\beta_2\wedge \theta^{n-1}_{k-1,q}).
	\end{align*}
	According to Corollary \ref{corollary:decomposition_theta_tau},
	\begin{align*}
	|z|^2\theta^n_{k,q}\equiv&\gamma_1\wedge[(n-k+q)\gamma_2\wedge\theta^{n-1}_{k,q}+(k-2q)\beta_2\wedge\theta^{n-1}_{k-1,q}]\\
	&+\beta_1\wedge[(k-2q)\gamma_2\wedge\theta^{n-1}_{k-1,q}+q\beta_2\wedge\theta^{n-1}_{k-2,q-1}]\quad\mod\omega_s.
	\end{align*}	
	As $\D$ vanishes on multiples of $\omega_s$, this implies
	\begin{align*}
	&\D(\phi(|z|^2)|z|^2\theta^n_{k,q})\\
	=&(n-k+q)\D(\phi(|z|^2)\gamma_1\wedge\gamma_2\wedge \theta^{n-1}_{k,q})+(k-2q)\D(\phi(|z|^2)\gamma_1\wedge\beta_2\wedge\theta^{n-1}_{k-1,q})\\
	&+(k-2q)\D(\phi(|z|^2)\beta_1\wedge\gamma_2\wedge\theta^{n-1}_{k-1,q})+q\D(\phi(|z|^2)\beta_1\wedge\beta_2\wedge\theta^{n-1}_{k-2,q-1})\\
	=&-(n-k+q)\D(\psi(|z|^2) \theta^{n}_{k,q})-\frac{1}{2}(k-2q)\D(\psi(|z|^2)\theta^{n}_{k,q})\\
	&+(k-2q)\D(\phi(|z|^2)\beta_1\wedge\gamma_2\wedge\theta^{n-1}_{k-1,q})+q\D(\phi(|z|^2)\beta_1\wedge\beta_2\wedge\theta^{n-1}_{k-2,q-1}).
	\end{align*}
	Rearranging this equation, we obtain the desired result.
\end{proof}

\begin{lemma}
	\label{lemma:replace_4_1-forms}
	Let $R\in\R[X,Y,Z]$ be a homogeneous polynomial of degree $n-2$ and $\phi\in C_c^\infty([0,\infty))$.
	Then
	\begin{align*}
	&2\D(\phi'(|z|^2)\gamma_1\wedge\gamma_2\wedge\beta_1\wedge\beta_2\wedge  R(\theta_0,\theta_1,\theta_2))\\
	=&\D(\phi(|z|^2)\left[\beta_1\wedge \gamma_2\wedge \theta_1 -2\beta_1\wedge\beta_2\wedge \theta_0\right]\wedge R(\theta_0,\theta_1,\theta_2)).
	\end{align*}
\end{lemma}
\begin{proof}
	 Using that $\D$ vanishes on closed forms, we obtain
	\begin{align*}
	&\D(\phi'(|z|^2)\gamma_1\wedge\gamma_2\wedge\beta_1\wedge\beta_2\wedge  R(\theta_0,\theta_1,\theta_2))\\
	=&\frac{1}{2}\D\left(d\phi(|z|^2)\gamma_2\wedge\beta_1\wedge\beta_2\wedge R(\theta_0,\theta_1,\theta_2)\right)\\
	=&\frac{1}{2}\D\left(d(\phi(|z|^2)\gamma_2\wedge\beta_1\wedge\beta_2\wedge R(\theta_0,\theta_1,\theta_2))\right)-\frac{1}{2}\left(\phi(|z|^2)d(\gamma_2\wedge\beta_1\wedge\beta_2\wedge R(\theta_0,\theta_1,\theta_2))\right)\\
	=&0-\frac{1}{2}\D\left(\phi(|z|^2)[2\theta_0\wedge\beta_1\wedge\beta_2-\gamma_2\wedge\omega_s\wedge\beta_2+\gamma_2\wedge\beta_1\wedge\theta_1]\wedge R(\theta_0,\theta_1,\theta_2)\right)\\
	=&\frac{1}{2}\D(\phi(|z|^2)\left[\beta_1\wedge \gamma_2\wedge \theta_1 -2\beta_1\wedge\beta_2\wedge \theta_0\right]\wedge R(\theta_0,\theta_1,\theta_2)),
	\end{align*}
	due to the fact that $\D$ vanishes on multiples of $\omega_s$.
\end{proof}

\section{Characterization of smooth valuations $\U(n)$-invariant valuations}
	\label{section:characSmoothVal}
	\subsection{Four families of Monge--Amp\`ere-type operators}
			The differential forms constructed so far give rise to the following Monge--Amp\`ere-type operators (where we use the notation established in \eqref{eq:MAOperators})
		\begin{definition}
			\label{definition:ComplexMA}
			We define measure-valued valuations $\Conv(\R^n,\R)\rightarrow\mathcal{M}(\C^n)$ by 
			\begin{align*}
				\Theta^n_{k,q}:=&c_{n,k,q}\Psi_{\theta^n_{k,q}}&&\text{for}~\max(0,k-n)\le q\le \left\lfloor\frac{k}{2}\right\rfloor,\\
				\mathcal{B}^n_{k,q}:=&c_{n,k,q}\Psi_{\beta_1\wedge\beta_2\wedge \theta^{n-1}_{k-2,q-1}} &&\text{for}~2\le k\le 2n-1,~\max(1,k-n)\le q\le \left\lfloor\frac{k}{2}\right\rfloor,\\
				\mathcal{C}^n_{k,q}:=&\frac{c_{n,k,q}}{2}\Psi_{\beta_1\wedge\gamma_2\wedge \theta^{n-1}_{k-1,q}} &&\text{for}~1\le k\le 2n-1,~\max(0,k-n)\le q\le \left\lfloor\frac{k-1}{2}\right\rfloor,\\
				\Upsilon^n_{k,q}:=&\mathcal{B}^n_{k,q}-\mathcal{C}^n_{k,q}&& \text{for}~2\le k\le 2n-1,~\max(1,k-n)\le q\le \left\lfloor\frac{k-1}{2}\right\rfloor,
			\end{align*}
			where we set $c_{n,k,q}:=\frac{1}{(n-k+q)!q!(k-2q)!}$. 
		\end{definition}
		For simplicity we set $\upsilon^n_{k,q}:=\beta_1\wedge \beta_2\wedge \theta^{n-1}_{k-2,q-1}-2\beta_1\wedge\gamma_2\wedge\theta^{n-1}_{k-1,q}$, so $\Upsilon^n_{k,q}=c_{n,k,q}\Psi_{\upsilon^n_{k,q}}$. In order to avoid unnecessary distinctions, we set $\Theta^n_{k,q}=0$, $\Upsilon^n_{k,q}=0$ if $n,k,q$ do not satisfy the conditions above.\\
	
		For $\max(0,k-n)\le q\le \lfloor\frac{k}{2}\rfloor$, we denote the coordinates on the two factors of $E_{k,q}=\C^q\times \R^{k-2q}$ by $(z,x)$, where $z=(z_1,\dots,z_q)$ and $x=(x_{q+1},\dots,x_{k-q})$ are the restrictions of the coordinates from $\C^n$ to $E_{k,q}\cong \C^q\times \R^{k-2q}\times\{0\}^{n-k+q}\subset\C^n$.
		\begin{proposition}
			\label{proposition:RestrictionExtremalSpaceMeasure}
			Let $\phi\in C_c(\C^n)$, $1\le k\le 2n$. Then
			\begin{enumerate}
				\item for $\max(0,k-n)\le p,q\le \lfloor\frac{k}{2}\rfloor$, and $f\in \Conv(E_{k,p},\R)$:
				\begin{align*}
					\int_{\C^n} \phi d\Theta^n_{k,q}(\pi_{E_{k,p}}^*f)=\delta_{p,q}\int_{E_{k,q}\times E_{k,q}^\perp} \phi(z) d(\MA_{E_{k,q}}(f)\otimes \vol_{E_{k,q}^\perp})(z),
				\end{align*}
				\item for $2\le k\le 2n-1$, $\max(1,k-n)\le p,q\le \lfloor\frac{k}{2}\rfloor$:
				\begin{align*}
					&\int_{\C^n} \phi d\mathcal{B}^n_{k,q}(\pi_{E_{k,p}}^*f)
					=\delta_{p,q} \int_{E_{k,q}\times E_{k,q}^\perp} \phi(z)\frac{\sum_{j=1} ^q|z_j|^2}{q} d(\MA_{E_{k,q}}(f)\otimes \vol_{E_{k,q}^\perp})(z)
				\end{align*}
				\item for $1\le k\le 2n-1$, $\max(1,k-n)\le p,q\le \lfloor\frac{k-1}{2}\rfloor$:
				\begin{align}
					\label{eq:RestcalC}
					\int_{\C^n} \phi d\mathcal{C}^n_{k,q}(\pi_{E_{k,p}}^*f)=\delta_{p,q}\int_{E_{k,p}\times E_{k,p}^\perp} \phi(z)\frac{\sum_{j=1}^{k-2q}|x_{q+j}|^2}{2(k-2q)} d(\MA_{E_{k,p}}(f)\otimes \vol_{E_{k,q}^\perp})(z),
				\end{align}
			\end{enumerate} 
			where $\vol_{E_{k,q}^\perp}$ denotes the Lebesgue measure on $E_{k,q}^\perp$.
		\end{proposition}
		\begin{proof}
			In each of these cases, both sides of the equation define continuous valuations in $\VConv_k(E_{k,q})$. In order to show that these valuations coincide, we calculate the Fourier--Laplace transform of their Goodey--Weil distribution. Since the calculation is very similar in each of these cases, we will show \eqref{eq:RestcalC} and leave the other two cases to the reader. Consider the valuations
			\begin{align*}
				\mu(f):=\int_{\C^n} \phi d\mathcal{C}^n_{k,q}(\pi_{E_{k,p}}^*f), \quad f\in \Conv(E_{k,p},\R).
			\end{align*}
			 Due to Lemma \ref{lemma:interpretation_fourier_maxdegree}, it is sufficient to evaluate $\mathcal{F}(\GW(\mu))[i\otimes w_1,\dots,i\otimes w_k]$ in the orthogonal bases given by
			\begin{align*}
				w_j=\begin{cases}
					a_l  e_l & j=2l-1, 1\le l\le p,\\
					b_l  ie_l & j=2l, 1\le l\le p,\\
					c_l  e_{p+l}& j=2p+l, 1\le l\le k-2p,
				\end{cases}
			\end{align*}  
			for $a_j,b_j,c_j\in\R$. Consider the function $f\in \Conv(\C^n,\R)$ given by
			\begin{align*}
				f(z)&=\sum_{j=1}^k \lambda_j \exp\left(\Re\langle w_j,z\rangle\right)\\
				&=\sum_{j=1}^p \lambda_{2j-1}\exp(a_jx_j)+\lambda_{2j}\exp(-b_jy_j)+\sum_{j=1}^{k-2p} \lambda_{2p+j}\exp(c_jx_{p+j}).
			\end{align*}
			Then $f$ is the pullback of the corresponding sum of exponential functions defined on $E_{k,p}$.
			Since $f$ is smooth, the differential cycle $D(f)$ is given by integration over the graph of $df$. Consider the map $G_f:\C^n\rightarrow \C^n\times\C^n$, $G_f(z)=(z,df(z))$. Then
			\begin{align}
				\label{eq:formulaCalCRestriction}
				\mathcal{F}(\GW(\mu))[ w_1\otimes i,\dots, w_k\otimes i]=\frac{(-1)^k}{k!}\frac{c_{n,k,q}}{2}\int_{\C^n}\phi(z)\frac{\partial}{\partial\lambda_1\dots\lambda_k}\Big|_0 G_f^*(\beta_1\wedge \gamma_2\wedge \theta^{n-1}_{k-1,q}).
			\end{align}
			
			A short calculation shows
			\begin{align*}
				G_f^*\theta_1=&\sum_{j=1}^p[\lambda_{2j-1}a_j^2\exp(a_jx_j)+\lambda_{2j}b_j^2\exp(-b_jy_j)]dx_j\wedge dy_j,\\
				&+\sum_{j=1}^{k-2p}\lambda_{2p+j}c_j^2\exp(c_jx_{p+j})dx_{p+j}\wedge dy_{p+j}\\
				=&:\theta_{1,c}+\theta_{1,r},\\
				G_f^*\theta_2=&\sum_{j=1}^p a_j^2b_j^2\lambda_{2j-1}\lambda_{2j} \exp(a_jx_j-b_jy_j)dx_j\wedge dy_j=:\theta_{2,c},\\
				G_f^*\beta_1=&\sum_{j=1}^p \lambda_{2j-1}a_j^2x_j\exp(a_jx_j)dx_j+\lambda_{2j}b_j^2y_j\exp(-b_jy_j)dy_j ,\\
				&+\sum_{j=1}^{k-2p} \lambda_{2p+j}c_j^2x_{p+j}\exp(c_jx_{p+j})dx_{p+j}.			
			\end{align*}
			Moreover, we split the forms
			\begin{align*}
				G_f^*\theta_0=&\sum_{j=1}^{k-p}dx_j\wedge dy_j+\sum_{j=k-p+1}^{n}dx_j\wedge dy_j=:\theta_{0,p}+\theta_{0,\perp}\\
				G_f^*\gamma_2=&\sum_{j=1}^{p}x_jdy_j-y_jdx_j+\sum_{j=p+1}^{k-p}x_jdy_j-y_jdx_j+\sum_{j=k-p+1}^{n}x_jdy_j-y_jdx_j.
			\end{align*}
			Note that $G_f^*\theta^{n-1}_{k-1,q}$ always involves $n-1$ products of the forms $dx_j\wedge dy_j$, $1\le j\le n$. In order to obtain a nontrivial $2n$-form on $\C^n$, the form $\beta_1\wedge \gamma_2$ has to contribute the missing form to this product, so
			\begin{align*}
				G_f^*(\beta_1\wedge\gamma_2\wedge \theta^{n-1}_{k-1,q})=(\delta_c+\delta_r)\wedge \theta_{0,\perp}^{n-k+q}\wedge (\theta_{1,c}+\theta_{1,r})^{k-1-2q}\wedge \theta_{2,c}^q
			\end{align*}
			for
			\begin{align*}
				\delta_c=&\sum_{j=1}^p\left(\lambda_{2j-1}a_j^2x_j^2\exp(a_jx_j)+\lambda_{2j}b_j^2y_j^2\exp(-b_jy_j)\right)dx_j\wedge dy_j,\\
				\delta_r=&\sum_{j=1}^{k-2p}\lambda_{2p+j}c_j^2x_{p+j}^2\exp(c_jx_{p+j})dx_{p+j}\wedge dy_{p+j}.
			\end{align*}
			Taking the degrees of these forms into account, we thus obtain
			\begin{align*}
				G_f^*(\beta_1\wedge\gamma_2\wedge \theta^{n-1}_{k-1,q})=&\binom{k-1-q}{2(p-q)-1}\delta_c\wedge \theta_{0,\perp}^{n-k+q}\wedge \theta_{1,c}^{2p-2q-1}\wedge\theta_{1,r}^{k-2p}\wedge \theta_{2,c}^q\\
				&+\binom{k-1-q}{2(p-q)}\delta_r\wedge \theta_{0,\perp}^{n-k+q}\wedge \theta_{1,c}^{2p-2q}\wedge\theta_{1,r}^{k-2p-1}\wedge \theta_{2,c}^q
			\end{align*}
			The first term vanishes obviously for $p\le q$, while for $p\ge q+1$, the form $\delta_c\wedge \theta_{1,c}^{2p-2q-1}\wedge \theta_{2,c}^q$ has degree $2(2p-q)\ge 2(2p-(p-1))=4p+2$ and involves $2p$ variables, so this term vanishes identically. Similarly, the second term vanishes for $p<q$, and for $p\ge q$, $\theta_{1,c}^{2p-2q}\wedge \theta_{2,c}^q$ is a form of degree $2(2p-q)$ in $2p$ variables. Thus the second form vanishes for $p\ne q$. For $p=q$, a simple calculation shows that
			\begin{align*}
				\theta_{2,c}^p=&p!\exp\left(\Re\left\langle\sum_{j=1}^{2p}w_j,z\right\rangle\right) 
				\prod_{j=1}^{p}a_j^2b_j^2\lambda_{2j-1}\lambda_{2j}dx_{j}\wedge dy_{j}&
			\end{align*}
			as well as
			\begin{align*}
				\delta_r\wedge \theta_{1,r}^{k-2p-1}=&(k-2p-1)!\exp\left(\Re\left\langle\sum_{j=2p+1}^{k}w_j,z\right\rangle\right)\left(\sum_{j=1}^{k-2p}|x_{p+j}|^2\right)\cdot\prod_{j=1}^{k-2p}c_j^2\lambda_{2p+j}dx_{p+j}\wedge dy_{p+j}.
			\end{align*}			
		Note that 
		\begin{align*}
			\theta_{0,\perp}^{n-k+q}=(n-k+q)!\prod_{j={k-q+1}}^{n}dx_j\wedge dy_j.
		\end{align*}
		Since $c_{n,k,q}=\frac{1}{(n-k+q)!q!(k-2q)!}$, this implies for $p=q$,
		\begin{align*}
			&\frac{c_{n,k,q}}{2}\frac{\partial^k}{\partial\lambda_1\dots\partial\lambda_k}\Big|_0G_f^*(\beta_1\wedge\gamma_2\wedge \theta^{n-1}_{k-1,q})\\
			=&\frac{1}{2(k-2q)}\exp\left(\Re\left\langle\sum_{j=1}^{k}w_j,z\right\rangle\right)\left(\sum_{j=1}^{{k-2q}}|x_{q+j}|^2\right) \det(\Re\langle w_j,w_l\rangle)_{j,l=1}^k\prod_{j=1}^ndx_j\wedge dy_j.
		\end{align*}
		From \eqref{eq:formulaCalCRestriction}, we thus obtain
		\begin{align*}
			&\mathcal{F}(\GW(\mu))[w_1\otimes i,\dots,w_k\otimes i]\\
			=&\frac{(-1)^k}{k!}\frac{\det(\Re\langle w_j,w_l\rangle)_{j,l=1}^k}{2(k-2q)}\int_{\C^n}\phi(z)\exp\left(\Re\left\langle\sum_{j=1}^{k}w_j,z\right\rangle\right)\left(\sum_{j=1}^{{k-2q}}|x_{q+j}|^2\right) d\vol_{2n}(z),
		\end{align*}
		which is exactly the Fourier--Laplace transform of the Goodey--Weil distribution corresponding to the right hand side of \eqref{eq:RestcalC}, compare Lemma \ref{lemma:interpretation_fourier_maxdegree}. Since both sides vanish for $p\ne q$, this completes the proof.
	\end{proof}

	\begin{corollary}
		\label{corollary:vanishing_propertiesMeasures}
		Let $0\le k\le 2n$, and $\max(0,k-n)\le p,q\le \lfloor\frac{k}{2}\rfloor$. If $p\ne q$, then $\Theta^n_{k,q}$, $\mathcal{B}^n_{k,q}$, $\mathcal{C}^n_{k,q}$ and $\Upsilon^n_{k,q}$ vanish on functions of the form $\pi_{E_{k,p}}^*f$ for $f\in\Conv(E_{k,p},\R)$.
	\end{corollary}
	\begin{proof}
		By Proposition \ref{proposition:RestrictionExtremalSpaceMeasure}, the integral of $\phi\in C_c(\C^n)$ with respect to any of the measures $\Theta^n_{k,q}(\pi_{E_{k,p}}^*f)$, $\mathcal{B}^n_{k,q}(\pi_{E_{k,p}}^*f)$, $\mathcal{C}^n_{k,q}(\pi_{E_{k,p}}^*f)$ and $\Upsilon^n_{k,q}(\pi_{E_{k,p}}^*f)$ for $f\in \Conv(E_{k,p},\R)$ vanishes, so these measures vanish identically.
	\end{proof}
	\begin{corollary}
		\label{corollary:measuresInduceVConkq}
		For every $\phi\in C^\infty_c([0,\infty))$ and every $\max(0,k-n)\le q\le \frac{k}{2}$, the valuations
		\begin{align*}
			f\mapsto \int_{\C^n}\phi(|z|^2)d\Theta^n_{k,q}(f;z), &&			f\mapsto \int_{\C^n}\phi(|z|^2)d\Upsilon^n_{k,q}(f;z)
		\end{align*}
		belong to $\VConv_{k,q}(\C^n)^{\U(n),sm}$.
	\end{corollary}
	\begin{proof}
		By construction, these valuations are induced by smooth differential forms, so they define smooth valuations. Moreover, since the relevant forms are $\U(n)$-invariant, so are the valuations. Finally, note that their restrictions to pullbacks of functions defined on $E_{k,p}$ along the orthogonal projection vanishes by Proposition \ref{proposition:RestrictionExtremalSpaceMeasure}. Thus they belong to $\VConv_{k,q}(\C^n)^{\U(n),sm}$.
	\end{proof}

		Next, we are going to explicitly calculate the restriction of these valuations to the extremal subspace $E_{k,q}$, $\max(0,k-n)\le q\le \lfloor\frac{k}{2}\rfloor$, which involves the Abel transform of the underlying densities. We refer to \cite[Chapter 13]{BracewellFouriertransformits1986} for a background on this transform and only collect the facts that are relevant for our purposes.\\
		Set $
		C_b((0,\infty)):=\{\phi\in C((0,\infty)): \supp\phi\subset (0,R]\quad\text{for some }R>0\}$.
		The Abel transform $\mathcal{A}:C_b((0,\infty))\rightarrow C_b((0,\infty))$ is given by
		\begin{align*}
			\mathcal{A}\phi (t):=\int_{-\infty}^{+\infty}\phi(\sqrt{t^2+s^2})ds\quad\text{for }t>0.
		\end{align*}
		In particular, if we consider the $k$-times iterated Abel transform $\mathcal{A}^k=\mathcal{A}\circ\dots\circ\mathcal{A}$, we obtain
		\begin{align*}
			\mathcal{A}^k\phi(t)=\int_{\R^k}\phi(\sqrt{t^2+|x|^2})d\vol_k(x).
		\end{align*}
		Note that $\mathcal{A}(\phi)$ is a smooth function if $\phi$ is smooth. If $\psi\in C_b((0,\infty))$ is continuously differentiable, then it is contained in the image of $\mathcal{A}$ and 
		\begin{align}
			\label{eq:inverseAbel}
			\mathcal{A}^{-1}\psi(t)=-\frac{1}{\pi}\int_t^\infty \frac{\psi'(s)}{\sqrt{s^2-t^2}}dt.
		\end{align}
		In particular, $\mathcal{A}:C_b((0,\infty))\rightarrow C_b((0,\infty))$ is injective. The Abel transform is intimately related to the Fourier--Laplace transform of rotation invariant functions: For $\phi\in C_c([0,\infty))$, $w\in E\in\Gr_k(\C^n)$
		\begin{align*}
			\mathcal{F}(\phi(|\cdot|))[w]=\int_{\C^n}\phi(|z|)\exp(i\langle w,z\rangle )d\vol_{2n}(z)=\int_E \mathcal{A}^{2n-k}\phi(|z'|)\exp(i\langle w,z'\rangle )d\vol_{E}(z').
		\end{align*}
		By considering the Fourier--Laplace transforms of the Goodey--Weil distributions of the smooth valuations calculated in Proposition \ref{proposition:RestrictionExtremalSpaceMeasure}, we obtain the following formulas for the restriction of smooth valuations to the extremal subspaces $E_{k,q}$.
		\begin{corollary}
			\label{corollary:DensitiesRestrictionExtremalSmoothCase}
			Let $1\le k\le 2n$, $\phi\in C_c([0,\infty))$. Then for every $f\in\Conv(E_{k,q},\R)$,
			\begin{align*}
				\int_{\C^n}\phi(|z|)d\Theta^n_{k,q}(\pi_{E_{k,q}}^*f;z)=&\int_{E_{k,q}} \mathcal{A}^{2n-k}\phi (|(z,x)|)d\MA_{E_{k,q}}(f;(z,x)),\\
				\int_{\C^n}\phi(|z|)d\mathcal{B}^n_{k,q}(\pi_{E_{k,q}}^*f;z)=&\frac{1}{q}\int_{E_{k,q}}|z|^2\mathcal{A}^{2n-k}\phi(|(z,x)|) d\MA_{E_{k,q}}(f;(z,x)),\\
				\int_{\C^n}\phi(|z|)d\mathcal{C}^n_{k,q}(\pi_{E_{k,q}}^*f;z)=&\frac{1}{2(k-2q)}\int_{E_{k,q}}|x|^2\mathcal{A}^{2n-k}\phi(|(z,x)|)d\MA_{E_{k,q}}(f;(z,x)),\\
				\int_{\C^n}\phi(|z|)d\Upsilon^n_{k,q}(\pi_{E_{k,q}}^*f;z)=&\int_{E_{k,q}}\left(\frac{|z|^2}{q}-\frac{|x|^2}{2(k-2q)}\right)\mathcal{A}^{2n-k}\phi(|(z,x)|)d\MA_{E_{k,q}}(f;(z,x)).
			\end{align*}
		\end{corollary}
		\begin{proof}
			These are all direct consequences of Proposition \ref{proposition:RestrictionExtremalSpaceMeasure}. As an example, we consider the third equation. Applying Fubini's Theorem to the expression in Proposition \ref{proposition:RestrictionExtremalSpaceMeasure}, we obtain
			\begin{align*}
				&\int_{\C^n} \phi(|z|)d\mathcal{C}^n_{k,q}(\pi_{E_{k,q}}^*f;z)\\
				=&\frac{1}{2(k-2q)}\int_{E_{k,q}} \left[\int_{E_{k,q}^\perp}\phi(|(z,x,w)|)d\vol_{2n-k}(w)\right]|z|^2 d\MA_{E_{k,q}}(f;(z,x))\\
				=&\frac{1}{2(k-2q)}\int_{E_{k,q}} \mathcal{A}^{2n-k}\phi(|(z,x)|)|z|^2 d\MA_{E_{k,q}}(f;(z,x)),
			\end{align*}
			which shows the desired formula. The other cases follow with the same calculation.
		\end{proof}
	
		\begin{corollary}
			\label{corollary:uniqueness_representation}
			Let $1\le k\le 2n$. For $\phi_q\in C_c([0,\infty))$, $\max(0,k-n)\le q\le \lfloor\frac{k}{2}\rfloor$, and $\psi_q\in C_c([0,\infty))$, $\max(1,k-n)\le q\le \lfloor\frac{k}{2}\rfloor$, let $\mu\in\VConv_k(\C^n)^{\U(n)}$ be given by
			\begin{align*}
				\mu(f)=\sum_{q=\max(0,k-n)}^{\lfloor\frac{k}{2}\rfloor}\int_{\C^n}\phi_q(|z|)d\Theta^n_{k,q}(f;z)+\sum_{q=\max(1,k-n)}^{\lfloor\frac{k}{2}\rfloor}\int_{\C^n}\psi_q(|z|)d\Upsilon^n_{k,q}(f;z)
			\end{align*}
			Then the following are equivalent:
			\begin{enumerate}
				\item $\mu$ vanishes identically.
				\item $\pi_{E_{k,q}*}\mu=0$ for every $\max(0,k-n)\le q\le \lfloor\frac{k}{2}\rfloor$.
				\item $\phi_q=0$ and $\psi_q=0$ for every $q$.
			\end{enumerate}
		\end{corollary}
		\begin{proof}
			The implications 3. $\Rightarrow$ 1. and 1. $\Rightarrow$ 2. are trivial. Let us show that 2. implies 3. Combining Corollary \ref{corollary:vanishing_propertiesMeasures} and Corollary \ref{corollary:DensitiesRestrictionExtremalSmoothCase}, we obtain the following for $f\in\Conv(E_{k,q},\R)$:
			\begin{itemize}
				\item for $\max(1,k-n)\le q\le \lfloor\frac{k}{2}\rfloor$:
				\begin{align*}
					\pi_{E_{k,q}*}\mu(f)=\int_{E_{k,q}} \mathcal{A}^{2n-k}\phi_q (|(z,x)|)+\mathcal{A}^{2n-k}\psi_q(|(z,x)|)\left(\frac{|z|^2}{q}-\frac{|x|^2}{2(k-2q)}\right)d\MA_{E_{k,q}}(f;(z,x)),
				\end{align*}
				\item for $q=0$ if $k\le n$, or $q=\frac{k}{2}$ if $k$ is even:
				\begin{align*}
					\pi_{E_{k,q}*}\mu(f)=\int_{E_{k,q}} \mathcal{A}^{2n-k}\phi_q (|(z,x)|)d\MA_{E_{k,q}}(f;(z,x)).
				\end{align*}
			\end{itemize}
			If 2. holds, then these expressions vanish identically, so Lemma \ref{lemma:interpretation_fourier_maxdegree} implies for $q=0$ if $k\le n$, or $q=\frac{k}{2}$ if $k$ is even that
			\begin{align*}
				\mathcal{A}^{2n-k}\phi_q (|(z,x)|)=0\quad\text{for}~(z,x)\in E_{k,q}.
			\end{align*}
			Thus $\mathcal{A}^{2n-k}\phi_q=0$ in this case, and so $\phi_q=0$ as $\mathcal{A}^{2n-k}$ injective.\\
			If $\max(1,k-n)\le q\le \lfloor\frac{k}{2}\rfloor$, then 
			\begin{align*}
				\mathcal{A}^{2n-k}\phi_q (|(z,x)|)+\mathcal{A}^{2n-k}\psi_q(|(z,x)|)\left(\frac{|z|^2}{q}-\frac{|x|^2}{2(k-2q)}\right)=0 \quad\text{for} (z,x)\in E_{k,p}.
			\end{align*}
			If we evaluate this equation in $(z,0)$, $z\in \C^q$ and $(0,x)$, $x\in \R^{k-2q}$, we obtain the following two equations for any $t\in[0,\infty)$:
			\begin{align*}
				\mathcal{A}^{2n-k}\phi_q (t)+\mathcal{A}^{2n-k}\psi_q(t)\frac{t^2}{q}&=0,\\
				\mathcal{A}^{2n-k}\phi_q (t)-\mathcal{A}^{2n-k}\psi_q(t)\frac{t^2}{2(k-2q)}&=0.
			\end{align*}
			It is easy to see that this implies $\mathcal{A}^{2n-k}\phi_q=	\mathcal{A}^{2n-k}\psi_q=0$, so $\phi_q=\psi_q=0$ since $\mathcal{A}^{2n-k}$ is injective.
		\end{proof}

	\subsection{Proofs of the main results}
	In this section we prove Theorem \ref{maintheorem:smooth_unitarily_invariant_valuations} and use it to obtain Theorem \ref{maintheorem:vanishing} and Theorem \ref{maintheorem:DecompSmoothCase}.
	
	\begin{proof}[Proof of Theorem \ref{maintheorem:smooth_unitarily_invariant_valuations}]
		Let us first establish that any smooth valuation in $\VConv_k(\C^n)^{\U(n)}$ admits such an integral representation. We may restrict ourselves to the case $k\ge 1$ since any $0$-homogeneous valuation is constant and may thus written in terms of $\Theta^n_{0,0}(f)$, which is just the Lebesgue measure for any $f\in\Conv(\C^n,\R)$.\\
		By Proposition \ref{proposition:representation_by_invariant_form}, any smooth $\mu\in\VConv_k(\C^n)^{\U(n)}$ of degree $k\ge 1$ may be represented by a differential form $\omega\in (\Omega_c^{n-k}(\C^n)\otimes\Lambda^k(\C^n)^*)^{\U(n)}$ of degree $2n$ . Theorem \ref{theorem:representation_invariant_differential_form_theta} shows that $\omega$ is given by
		\begin{align*}
			\omega=&\gamma_1\wedge\gamma_2\wedge\beta_1\wedge\beta_2\wedge R(\theta_0,\theta_1,\theta_2,\omega_s)\left[|z|^2\right]\\
			&+\gamma_1\wedge\gamma_2 \wedge R_{\gamma_1\gamma_2}(\theta_0,\theta_1,\theta_2,\omega_s)\left[|z|^2\right]+\gamma_1\wedge\beta_1\wedge R_{\gamma_1\beta_1}(\theta_0,\theta_1,\theta_2,\omega_s)\left[|z|^2\right]\\
			&+\gamma_1\wedge\beta_2\wedge R_{\gamma_1\beta_2}(\theta_0,\theta_1,\theta_2,\omega_s)\left[|z|^2\right]+\gamma_2\wedge\beta_1\wedge R_{\gamma_2\beta_1}(\theta_0,\theta_1,\theta_2,\omega_s)\left[|z|^2\right]\\
			&+\gamma_2\wedge\beta_2\wedge R_{\gamma_2\beta_2}(\theta_0,\theta_1,\theta_2,\omega_s)\left[|z|^2\right]+\beta_1\wedge\beta_2 \wedge R_{\beta_1\beta_2}(\theta_0,\theta_1,\theta_2,\omega_s)\left[|z|^2\right]\\
			&+R'(\theta_0,\theta_1,\theta_2,\omega_s)\left[|z|^2\right]
		\end{align*}
		for suitable $R,R_{\gamma_1\gamma_2},\dots,R'\in C^\infty([0,\infty),\R[X,Y,Z,W])$. If the support of $\omega$ is contained in $B_R(0)\times\C^n\subset T^*\C^n$, we may choose $\psi\in C_c^\infty([0,\infty),[0,1])$ with $\psi\equiv 1$ on $[0,R^2]$. By multiplying the equation above with $\psi(|\cdot|^2)$, we may thus assume that the support of every coefficient of these polynomials is also bounded. As the differential cycle vanishes on multiples of $\omega_s$, we may further assume that these polynomials do not depend on $\omega_s$. Then according to Lemma \ref{lemma:Rumin_differential_2_1-forms},
		\begin{align*}
			\D\omega=&\D\left(\gamma_1\wedge\gamma_2\wedge\beta_1\wedge\beta_2\wedge R(\theta_0,\theta_1,\theta_2 )\left[|z|^2\right]\right)\\
			&+\D\left(\gamma_1\wedge\gamma_2 \wedge R_{\gamma_1\gamma_2}(\theta_0,\theta_1,\theta_2 )\left[|z|^2\right]\right)+0\\
			&+\D\left(\gamma_1\wedge\beta_2\wedge R_{\gamma_1\beta_2}(\theta_0,\theta_1,\theta_2 )\left[|z|^2\right]\right)+\D\left(\gamma_2\wedge\beta_1\wedge R_{\gamma_2\beta_1}(\theta_0,\theta_1,\theta_2 )\left[|z|^2\right]\right)\\
			&+0+\D\left(\beta_1\wedge\beta_2 \wedge R_{\beta_1\beta_2}(\theta_0,\theta_1,\theta_2 )\left[|z|^2\right]\right)\\
			&+\D\left(R'(\theta_0,\theta_1,\theta_2 )\left[|z|^2\right]\right).
		\end{align*}
		Using  Lemma \ref{lemma:Rumin_differential_2_1-forms}, we may replace the terms involving $\gamma_1\wedge\gamma_2$ and $\gamma_1\wedge\beta_2$ by a suitable function in $C^\infty_c([0,\infty)),\R[X,Y,Z])$. Similarly, Lemma \ref{lemma:replace_4_1-forms} lets us replace the term involving $\gamma_1\wedge\gamma_2\wedge\beta_1\wedge\beta_2$ by multiples of $\beta_1\wedge\gamma_2$ and $\beta_1\wedge\beta_2$. Thus there exist $R_1,R_2,R_3\in C_c^\infty([0,\infty),\R[X,Y,Z])$ such that
		\begin{align*}
			\D\omega=&\D (R_1(\theta_0,\theta_1,\theta_2)\left[|z|^2\right])+\D(\beta_1\wedge\gamma_2\wedge R_2(\theta_0,\theta_1,\theta_2)\left[|z|^2\right])\\
			&+\D(\beta_1\wedge\beta_2\wedge R_3(\theta_0,\theta_1,\theta_2)\left[|z|^2\right]).
		\end{align*}
		The first term can clearly be expressed in terms of the forms $\theta^n_{k,q}$, while the last two terms can be written in terms of the forms
		\begin{align*}
			&\beta_1\wedge \beta_2\wedge \theta^{n-1}_{k-2,q-1}-2\beta_1\wedge\gamma_2\wedge\theta^{n-1}_{k-1,q}=\upsilon^n_{k,q},\\
			&(k-2q)\beta_1\wedge\gamma_2\wedge\theta^{n-1}_{k-1,q}+q\beta_1\wedge\beta_2\wedge \theta^{n-1}_{k-2,q-1},
		\end{align*}
		where we may replace the second form by a term involving $\theta^{n}_{k,q}$ due to Lemma \ref{lemma:Rumin_differential_pointwise_relation}. In the extremal cases $q=0$ or $q=\lfloor\frac{k}{2}\rfloor$, both of the forms coincide and we may replace them by a term involving $\theta^n_{k,q}$ due to Lemma \ref{lemma:Rumin_differential_pointwise_relation}. We thus find $\phi_{q}\in C^\infty_c([0,\infty))$ for $\max(0,k-n)\le q\le\lfloor\frac{k}{2}\rfloor$ and $\psi_q\in C^\infty_c([0,\infty))$ for $\max(1,k-n)\le q\le \lfloor\frac{k-1}{2}\rfloor$ such that
		\begin{align*}
			\D\omega=&\sum_{q=\max(0,k-n)}^{\lfloor\frac{k}{2}\rfloor}\D(\phi_q(|z|^2)\theta^n_{k,q})+\sum_{q=\max(1,k-n)}^{\lfloor\frac{k-1}{2}\rfloor}\D(\psi_q(|z|^2)\upsilon^n_{k,q}).
		\end{align*}
		The Kernel Theorem \ref{theorem:Kernel_theorem} thus implies for all $f\in\Conv(\C^n,\R)$
		\begin{align}
			\label{equation:integralrep}
			\mu(f)=D(f)[\omega]=\sum_{q=\max(0,k-n)}^{\lfloor\frac{k}{2}\rfloor}D(f)[\phi_q(|z|^2)\theta^n_{k,q}]+\sum_{q=\max(1,k-n)}^{\lfloor\frac{k-1}{2}\rfloor}D(f)\left[\psi_q(|z|^2)\upsilon^n_{k,q}\right],
		\end{align}
		which up to multiplication with the constants $c_{n,k,q}$ is precisely the desired representation, compare \eqref{eq:MAOperators}. \\
		
		In order to see that such an integral representation is unique, note that it is sufficient to show that \eqref{equation:integralrep} vanishes identically if and only if $\phi_q=0$ and $\psi_q=0$ for every $q$. However, this is the case due to Corollary \ref{corollary:uniqueness_representation}.
	\end{proof}
	In \cite{Knoerrgeometricdecompositionunitarily2024}, we will need the following support restrictions on the densities.
	\begin{proposition}
		Let $1\le k\le 2n$. Let $\mu\in \VConv_k(\C^n)^{\U(n)}$ be the  smooth valuation given by
		\begin{align*}
			\mu(f)=\sum_{q=\max(0,k-n)}^{\lfloor\frac{k}{2}\rfloor}\int_{\C^n}\phi_q(|z|^2)d\Theta^n_{k,q}(f;z)+\sum_{q=\max(1,k-n)}^{\lfloor\frac{k-1}{2}\rfloor}\int_{\C^n}\psi_q(|z|^2)d\Upsilon^n_{k,q}(f;z)
		\end{align*}
		for $\phi_{q}\in C^\infty_c([0,\infty))$, $\max(0,k-n)\le q\le\lfloor\frac{k}{2}\rfloor$, and $\psi_q\in C^\infty_c([0,\infty))$, $\max(1,k-n)\le q\le \lfloor\frac{k-1}{2}\rfloor$. For $R>0$ the following are equivalent:
		\begin{enumerate}
			\item $\supp\mu\subset B_R(0)$.
			\item $\supp\phi_q,\supp\psi_q\subset [0,\sqrt{R}]$ for every $q$.
		\end{enumerate}
	\end{proposition}
	\begin{proof}
		The implication 2. $\Rightarrow$ 1. is trivial. Let us show that 1. implies 2. Set $\tilde{\phi}_q(t):=\phi_q(t^2)$ and $\tilde{\psi}_q(t)=\psi_q(t^2)$ for $t\ge 0$. As in the proof of  Corollary \ref{corollary:uniqueness_representation}, we have for $f\in\Conv(E_{k,q},\R)$:
		\begin{itemize}
			\item for $\max(1,k-n)\le q\le \lfloor\frac{k}{2}\rfloor$:
			\begin{align*}
				\pi_{E_{k,q}*}\mu(f)=\int_{E_{k,q}} \mathcal{A}^{2n-k}\tilde{\phi}_q (|\cdot|)+\mathcal{A}^{2n-k}\psi_q(|\cdot|)\left(\frac{|z|^2}{q}-\frac{|x|^2}{2(k-2q)}\right)d\MA_{E_{k,q}}(f),
			\end{align*}
			\item for $q=0$ if $k\le n$, or $q=\frac{k}{2}$ if $k$ is even:
			\begin{align*}
				\pi_{E_{k,q}*}\mu(f)=\int_{E_{k,q}} \mathcal{A}^{2n-k}\tilde{\phi}_q (|\cdot|)d\MA_{E_{k,q}}(f).
			\end{align*}
		\end{itemize}
		Since $\pi_{E_{k,q}*}\mu\in\VConv_k(E_{k,q})$ is a valuation of top degree which is supported on $B_R(0)\subset E_{k,q}$ by Proposition \ref{proposition:behavior-support-pushforward}, Corollary \ref{corollary:support-top-degree} implies that the corresponding densities are supported on $B_R(0)$ as well, i.e. the functions given 
			\begin{itemize}
			\item for $\max(1,k-n)\le q\le \lfloor\frac{k}{2}\rfloor$ by
			\begin{align}
				\label{eq:densityRestriction}
			(z,x)\mapsto \mathcal{A}^{2n-k}\tilde{\phi}_q (|(z,x)|)+\mathcal{A}^{2n-k}\psi_q(|(z,x)|)\left(\frac{|z|^2}{q}-\frac{|x|^2}{2(k-2q)}\right),
			\end{align}
			\item and for $q=0$ if $k\le n$, or $q=\frac{k}{2}$ if $k$ is even by
			\begin{align*}
			(z,x)\mapsto\mathcal{A}^{2n-k}\tilde{\phi}_q (|(z,x)|)
			\end{align*}
		\end{itemize}
		are supported on $B_R(0)$. In the second case, this implies that $\mathcal{A}^{2n-k}\tilde{\phi}_q$ is supported on $[0,R]$. Since this function is smooth, we may use the formula for the inverse Abel transform in \eqref{eq:inverseAbel} to see that $\tilde{\phi}_q$ is supported on $[0,R]$. Thus $\phi_q$ is supported on $[0,\sqrt{R}]$.\\
		In the first case, we evaluate \eqref{eq:densityRestriction} in $(z,0)\in E_{k,q}$ and $(0,x)\in E_{k,q}$ to obtain the equations
		\begin{align*}
			\mathcal{A}^{2n-k}\tilde{\phi}_q (t)+\mathcal{A}^{2n-k}\tilde{\psi}_q(t)\frac{t^2}{q}&=0,\quad t>R\\
			\mathcal{A}^{2n-k}\tilde{\phi}_q (t)-\mathcal{A}^{2n-k}\tilde{\psi}_q(t)\frac{t^2}{2(k-2q)}&=0,\quad t>R.
		\end{align*}
		We may subtract the two equations to obtain $\mathcal{A}^{2n-k}\tilde{\psi}_q(t)=0$ for $t>R$, which implies $\mathcal{A}^{2n-k}\tilde{\phi}(t)=0$ for $t>R$. As before, we may use the inverse Abel transform in \eqref{eq:inverseAbel} to see that $\tilde{\phi}_q$ and $\tilde{\psi}_q$ are supported on $[0,R]$, so $\phi_q$ and $\psi_q$ are supported on $[0,\sqrt{R}]$.
	\end{proof}
	Next, we prove Theorem \ref{maintheorem:DecompSmoothCase} in the following form.
	\begin{theorem}
		\label{theorem:DirectSumSmooth}
		Let $0\le k\le 2n$. We have a direct sum decomposition 
		\begin{align*}
			\VConv_k(\C^n)^{\U(n),sm}=\bigoplus_{q=\max(0,k-n)}^{\lfloor\frac{k}{2}\rfloor}\VConv_{k,q}(\C^n)^{\U(n),sm}.
		\end{align*}
		In particular, for $\max(0,k-n)\le p,q\le\lfloor\frac{k}{2}\rfloor$ with $p\ne q$:
		\begin{align*}
			\VConv_{k,p}(\C^n)^{\U(n),sm}\cap\VConv_{k,q}(\C^n)^{\U(n),sm}=0.
		\end{align*} 		
	\end{theorem}
	\begin{proof}
		For $k\in\{0,1,2n-1,2n\}$, we have $\VConv_k(\C^n)^{\U(n),sm}=\VConv_{k,q}(\C^n)^{\U(n),sm}$ for $q=\max(0,k-n)=\lfloor\frac{k}{2}\rfloor$, so there is nothing to prove. We may thus in particular assume that $2\le k\le 2n-2$.\\
		We will first show that the sum of the spaces $\VConv_{k,q}(\C^n)^{\U(n),sm}$ is direct. Let $p\ne q$ and $\mu\in\VConv_{k,p}(\C^n)^{\U(n),sm}\cap \VConv_{k,q}(\C^n)^{\U(n),sm}$ be a valuation. We have to show that $\mu=0$.\\
		Theorem \ref{maintheorem:smooth_unitarily_invariant_valuations} implies that there exist $\phi_{q}\in C^\infty_c([0,\infty))$ for $\max(0,k-n)\le q\le\lfloor\frac{k}{2}\rfloor$ and $\psi_q\in C^\infty_c([0,\infty))$ for $\max(1,k-n)\le q\le \lfloor\frac{k-1}{2}\rfloor$ such that
		\begin{align}
			\label{eq:repSmoothVal}
			\mu(f)=\sum_{q=\max(0,k-n)}^{\lfloor\frac{k}{2}\rfloor}\int_{\C^n}\phi_q(|z|^2)d\Theta^n_{k,q}(f;z)+\sum_{q=\max(1,k-n)}^{\lfloor\frac{k-1}{2}\rfloor}\int_{\C^n}\psi_q(|z|^2)d\Upsilon^n_{k,q}(f;z)
		\end{align}
		for all $f\in\Conv(\C^n,\R)$. Since $\mu\in\VConv_{k,p}(\C^n)^{\U(n),sm}\cap \VConv_{k,q}(\C^n)^{\U(n),sm}$, we have $\pi_{e_{k,q'}*}\mu=0$ for every $\max(0,k-n)\le q'\le\lfloor\frac{k}{2}\rfloor$, so Corollary \ref{corollary:uniqueness_representation} shows that $\mu=0$.\\
		It remains to see that  $\VConv_k(\C^n)^{\U(n),sm}=\bigoplus_{q=\max(0,k-n)}^{\lfloor\frac{k}{2}\rfloor}\VConv_{k,q}(\C^n)^{\U(n),sm}$. First, the right hand side is obviously contained in the left hand side. For the other inclusion, let $\mu\in\VConv_k(\C^n)^{\U(n),sm}$ be given. We may again use Theorem \ref{maintheorem:smooth_unitarily_invariant_valuations} to see that $\mu$ admits a representation as in \eqref{eq:repSmoothVal}.		
		Since the valuations
		\begin{align*}
			&f\mapsto \int_{\C^n}\phi_q(|z|^2)d\Theta^n_{k,q}(f;z), &&f\mapsto \int_{\C^n}\psi_q(|z|^2)d\Upsilon^n_{k,q}(f;z)
		\end{align*}
		belong to $\VConv_{k,q}(\C^n)^{\U(n),sm}$ by Corollary \ref{corollary:measuresInduceVConkq}, we obtain $\mu\in \bigoplus_{q=\max(0,k-n)}^{\lfloor\frac{k}{2}\rfloor}\VConv_{k,q}(\C^n)^{\U(n),sm}$, which completes the proof.
	\end{proof}

Finally, we obtain Theorem \ref{maintheorem:vanishing} from the classification of smooth $\U(n)$-invariant valuations by approximation. 

\begin{proof}[Proof of Theorem \ref{maintheorem:vanishing}]
	Let $\mu_0\in\VConv_k(\C^n)^{\U(n)}$ and assume that
	\begin{align*}
		\pi_{E_{k,q}*}\mu_0=0
	\end{align*}
	for all $0\le \max(0,k-n)\le q\le \lfloor\frac{k}{2}\rfloor$. We have to show that $\mu_0=0$. This is trivial for $k=0$, since $\mu_0$ is constant in this case, so we may assume that $1\le k\le 2n$.\\
	
	Let $V_{k,q}\subset \Gr_k(\C^n)$ denote the orbit of $E_{k,q}$ under $\U(n)$ and set
	\begin{align*}
		W:=\left\{\mu\in\VConv_k(\C^n):\pi_{E*}\mu=0~\text{for every}~E\in V_{k,q},\max(0,k-n)\le q\le \left\lfloor\frac{k}{2}\right\rfloor\right\}.
	\end{align*}
	This is a $\U(n)$- and translation invariant closed subspace of $\VConv_k(\C^n)$. Moreover, since $\mu_0$ is $\U(n)$-invariant, $\mu_0\in W$.\\
	
	Since $W$ is closed and translation invariant, Theorem \ref{theorem:SquentialDensitySmoothVal} implies that there exists a sequence $(\mu_j)_j$ of smooth valuations which belong to $W$ and converge to $\mu_0$. Define
	\begin{align*}
		\tilde{\mu}_j(f):=\int_{\U(n)}\mu_j(f\circ g)dg,
	\end{align*} 
	where we integrate with respect to the Haar probability measure. Then as in the proof of \cite[Proposition 6.6]{KnoerrSmoothvaluationsconvex2024}, $(\tilde{\mu}_j)_j$ is a sequence of smooth $\U(n)$-invariant valuations converging to $\mu_0$. If $E\in V_{k,q}$ for some $\max(0,k-n)\le q\le \lfloor \frac{k}{2}\rfloor$ and $f\in \Conv(E,\R)$, then $\pi_E^*f\circ g=\pi_{g^{-1}E}^*(f\circ g)$ where $f\circ g\in\Conv(g^{-1}E,\R)$ with $g^{-1}E\in V_{k,q}$ for every $g\in\U(n)$, and so 
	\begin{align*}
		\tilde{\mu}_j(\pi_E^*f):=\int_{\U(n)}\mu_j(\pi_{g^{-1}E}^*(f\circ g))dg=0,
	\end{align*}
	because $\mu_j\in W$. In particular, $\tilde{\mu}_j$ is a smooth valuation with $\pi_{E_{k,q}*}\tilde{\mu}_j=0$ for every $\max(0,k-n)\le q\le \lfloor\frac{k}{2}\rfloor$. Theorem \ref{maintheorem:smooth_unitarily_invariant_valuations} and Corollary \ref{corollary:uniqueness_representation} thus show that $\tilde{\mu}_j$ vanishes identically. As $(\tilde{\mu}_j)_j$ converges to $\mu_0$, this implies $\mu_0=0$, which completes the proof.
\end{proof}
\bibliographystyle{abbrv}
\bibliography{../../../library/library.bib}

\begin{thebibliography}{10}

\bibitem{AleskerP.McMullensconjecture2000}
S.~Alesker.
\newblock On {P}. {M}c{M}ullen's conjecture on translation invariant
  valuations.
\newblock {\em Adv. Math.}, 155(2):239--263, 2000.

\bibitem{AleskerDescriptiontranslationinvariant2001}
S.~Alesker.
\newblock Description of translation invariant valuations on convex sets with
  solution of {P}. {M}c{M}ullen's conjecture.
\newblock {\em Geom. Funct. Anal.}, 11(2):244--272, 2001.

\bibitem{AleskerHardLefschetztheorem2004}
S.~Alesker.
\newblock Hard {L}efschetz theorem for valuations and related questions of
  integral geometry.
\newblock In {\em Geometric aspects of functional analysis}, volume 1850 of
  {\em Lecture Notes in Math.}, pages 9--20. Springer, Berlin, 2004.

\bibitem{AleskerValuationsconvexsets2005}
S.~Alesker.
\newblock Valuations on convex sets, non-commutative determinants, and
  pluripotential theory.
\newblock {\em Adv. Math.}, 195(2):561--595, 2005.

\bibitem{AleskerValuationsconvexfunctions2019}
S.~Alesker.
\newblock Valuations on convex functions and convex sets and {M}onge-{A}mp\`ere
  operators.
\newblock {\em Adv. Geom.}, 19(3):313--322, 2019.

\bibitem{AleskerFaifmanConvexvaluationsinvariant2014}
S.~Alesker and D.~Faifman.
\newblock Convex valuations invariant under the {L}orentz group.
\newblock {\em J. Differential Geom.}, 98(2):183--236, 2014.

\bibitem{BernigHadwigertypetheorem2009}
A.~Bernig.
\newblock A {H}adwiger-type theorem for the special unitary group.
\newblock {\em Geom. Funct. Anal.}, 19(2):356--372, 2009.

\bibitem{BernigIntegralgeometry$G_2$2011}
A.~Bernig.
\newblock Integral geometry under {$G_2$} and {$\mathrm{Spin}(7)$}.
\newblock {\em Israel J. Math.}, 184:301--316, 2011.

\bibitem{BernigInvariantvaluationsquaternionic2012}
A.~Bernig.
\newblock Invariant valuations on quaternionic vector spaces.
\newblock {\em J. Inst. Math. Jussieu}, 11(3):467--499, 2012.

\bibitem{BernigFaifmanValuationtheoryindefinite2017}
A.~Bernig and D.~Faifman.
\newblock Valuation theory of indefinite orthogonal groups.
\newblock {\em J. Funct. Anal.}, 273(6):2167--2247, 2017.

\bibitem{BernigEtAlCurvaturemeasurespseudo2022}
A.~Bernig, D.~Faifman, and G.~Solanes.
\newblock Curvature measures of pseudo-{R}iemannian manifolds.
\newblock {\em J. Reine Angew. Math.}, 788:77--127, 2022.

\bibitem{BernigFuHermitianintegralgeometry2011}
A.~Bernig and J.~H.~G. Fu.
\newblock Hermitian integral geometry.
\newblock {\em Ann. of Math. (2)}, 173(2):907--945, 2011.

\bibitem{BernigEtAlIntegralgeometrycomplex2014}
A.~Bernig, J.~H.~G. Fu, and G.~Solanes.
\newblock Integral geometry of complex space forms.
\newblock {\em Geom. Funct. Anal.}, 24(2):403--492, 2014.

\bibitem{BernigSolanesKinematicformulasquaternionic2017}
A.~Bernig and G.~Solanes.
\newblock Kinematic formulas on the quaternionic plane.
\newblock {\em Proc. Lond. Math. Soc. (3)}, 115(4):725--762, 2017.

\bibitem{BernigVoideSpininvariantvaluations2016}
A.~Bernig and F.~Voide.
\newblock Spin-invariant valuations on the octonionic plane.
\newblock {\em Israel J. Math.}, 214(2):831--855, 2016.

\bibitem{BobkovEtAlQuermassintegralsquasiconcave2014}
S.~G. Bobkov, A.~Colesanti, and I.~Fragal\`a.
\newblock Quermassintegrals of quasi-concave functions and generalized
  {P}r\'{e}kopa-{L}eindler inequalities.
\newblock {\em Manuscripta Math.}, 143(1-2):131--169, 2014.

\bibitem{BracewellFouriertransformits1986}
R.~N. Bracewell.
\newblock {\em The {F}ourier transform and its applications}.
\newblock McGraw-Hill Series in Electrical Engineering. Circuits and Systems.
  McGraw-Hill Book Co., New York, third edition, 1986.

\bibitem{CavallinaColesantiMonotonevaluationsspace2015}
L.~Cavallina and A.~Colesanti.
\newblock Monotone valuations on the space of convex functions.
\newblock {\em Anal. Geom. Metr. Spaces}, 3(1):167--211, 2015.

\bibitem{ColesantiLombardiValuationsspacequasi2017}
A.~Colesanti and N.~Lombardi.
\newblock Valuations on the space of quasi-concave functions.
\newblock In {\em Geometric aspects of functional analysis}, volume 2169 of
  {\em Lecture Notes in Math.}, pages 71--105. Springer, Cham, 2017.

\bibitem{ColesantiEtAlTranslationinvariantvaluations2018}
A.~Colesanti, N.~Lombardi, and L.~Parapatits.
\newblock Translation invariant valuations on quasi-concave functions.
\newblock {\em Studia Math.}, 243(1):79--99, 2018.

\bibitem{ColesantiEtAlMinkowskivaluationsconvex2017}
A.~Colesanti, M.~Ludwig, and F.~Mussnig.
\newblock Minkowski valuations on convex functions.
\newblock {\em Calc. Var. Partial Differential Equations}, 56(6):Paper No. 162,
  29, 2017.

\bibitem{ColesantiEtAlHessianvaluations2020}
A.~Colesanti, M.~Ludwig, and F.~Mussnig.
\newblock Hessian valuations.
\newblock {\em Indiana Univ. Math. J.}, 69(4):1275--1315, 2020.

\bibitem{ColesantiEtAlhomogeneousdecompositiontheorem2020}
A.~Colesanti, M.~Ludwig, and F.~Mussnig.
\newblock A homogeneous decomposition theorem for valuations on convex
  functions.
\newblock {\em J. Funct. Anal.}, 279(5):Paper No. 108573, 25, 2020.

\bibitem{ColesantiEtAlHadwigertheoremconvex2022}
A.~Colesanti, M.~Ludwig, and F.~Mussnig.
\newblock The {H}adwiger theorem on convex functions, {III}: {S}teiner formulas
  and mixed {M}onge-{A}mp\`ere measures.
\newblock {\em Calc. Var. Partial Differential Equations}, 61(5):Paper No. 181,
  37, 2022.

\bibitem{ColesantiEtAlHadwigertheoremconvex2023}
A.~Colesanti, M.~Ludwig, and F.~Mussnig.
\newblock The {H}adwiger theorem on convex functions, {IV}: {T}he {K}lain
  approach.
\newblock {\em Adv. Math.}, 413:Paper No. 108832, 2023.

\bibitem{ColesantiEtAlHadwigertheoremconvex2024}
A.~Colesanti, M.~Ludwig, and F.~Mussnig.
\newblock The {H}adwiger theorem on convex functions, {I}.
\newblock {\em Geom. Funct. Anal.}, 34(6):1839--1898, 2024.

\bibitem{ColesantiEtAlHadwigertheoremconvex2025}
A.~Colesanti, M.~Ludwig, and F.~Mussnig.
\newblock The {H}adwiger theorem on convex functions, {II}: {C}auchy-{K}ubota
  formulas.
\newblock {\em Amer. J. Math.}, 147(4):927--955, 2025.

\bibitem{ColesantiEtAlclassinvariantvaluations2020}
A.~Colesanti, D.~Pagnini, P.~Tradacete, and I.~Villanueva.
\newblock A class of invariant valuations on {$\mathrm{Lip}(S^{n-1})$}.
\newblock {\em Adv. Math.}, 366:Paper No. 107069, 37, 2020.

\bibitem{ColesantiEtAlContinuousvaluationsspace2021}
A.~Colesanti, D.~Pagnini, P.~Tradacete, and I.~Villanueva.
\newblock Continuous valuations on the space of {L}ipschitz functions on the
  sphere.
\newblock {\em J. Funct. Anal.}, 280(4):Paper No. 108873, 43, 2021.

\bibitem{FuMongeAmperefunctions.1989}
J.~H.~G. Fu.
\newblock Monge-{A}mp\`ere functions. {I}.
\newblock {\em Indiana Univ. Math. J.}, 38(3):711--743, 1989.

\bibitem{GoodeyWeilDistributionsvaluations1984}
P.~Goodey and W.~Weil.
\newblock Distributions and valuations.
\newblock {\em Proc. London Math. Soc. (3)}, 49(3):504--516, 1984.

\bibitem{HadwigerVorlesungenuberInhalt1957}
H.~Hadwiger.
\newblock {\em Vorlesungen \"{u}ber {I}nhalt, {O}berfl\"{a}che und
  {I}soperimetrie}.
\newblock Springer, Berlin-G\"{o}ttingen-Heidelberg, 1957.

\bibitem{HofstaetterKnoerrEquivariantEndomorphismsConvex2023}
G.~C. Hofst\"atter and J.~Knoerr.
\newblock Equivariant endomorphisms of convex functions.
\newblock {\em J. Funct. Anal.}, 285(1):Paper No. 109922, 39, 2023.

\bibitem{HuybrechtsComplexgeometry2005}
D.~Huybrechts.
\newblock {\em Complex geometry}.
\newblock Universitext. Springer-Verlag, Berlin, 2005.

\bibitem{KlainshortproofHadwigers1995}
D.~A. Klain.
\newblock A short proof of {H}adwiger's characterization theorem.
\newblock {\em Mathematika}, 42(2):329--339, 1995.

\bibitem{Knoerrsupportduallyepi2021}
J.~Knoerr.
\newblock The support of dually epi-translation invariant valuations on convex
  functions.
\newblock {\em J. Funct. Anal.}, 281(5):Paper No. 109059, 52, 2021.

\bibitem{Knoerrgeometricdecompositionunitarily2024}
J.~Knoerr.
\newblock A geometric decomposition for unitarily invariant valuations on
  convex functions.
\newblock {\em arXiv:2408.01352}, 2024.

\bibitem{KnoerrMongeAmpereoperators2024}
J.~Knoerr.
\newblock Monge-{A}mp\`ere operators and valuations.
\newblock {\em Calc. Var. Partial Differential Equations}, 63(4):Paper No. 89,
  34, 2024.

\bibitem{KnoerrSmoothvaluationsconvex2024}
J.~Knoerr.
\newblock Smooth valuations on convex functions.
\newblock {\em J. Differential Geom.}, 126(2):801--835, 2024.

\bibitem{KnoerrPaleyWienerSchwartz2025}
J.~Knoerr.
\newblock A {P}aley--{W}iener--{S}chwartz theorem for smooth valuations on
  convex functions.
\newblock {\em arXiv:2505.22464}, 2025.

\bibitem{KoneValuationsOrliczspaces2014}
H.~Kone.
\newblock Valuations on {O}rlicz spaces and {$L^\phi$}-star sets.
\newblock {\em Adv. in Appl. Math.}, 52:82--98, 2014.

\bibitem{KotrbatyIntegralgeometryoctonionic2020}
J.~Kotrbat{\'y}.
\newblock {\em Integral geometry on the octonionic plane}.
\newblock PhD thesis, Friedrich-Schiller-Universit{\"a}t Jena, 2020.

\bibitem{LudwigFisherinformationmatrix2011}
M.~Ludwig.
\newblock Fisher information and matrix-valued valuations.
\newblock {\em Adv. Math.}, 226(3):2700--2711, 2011.

\bibitem{LudwigValuationsSobolevspaces2012}
M.~Ludwig.
\newblock Valuations on {S}obolev spaces.
\newblock {\em Amer. J. Math.}, 134(3):827--842, 2012.

\bibitem{LudwigCovariancematricesvaluations2013}
M.~Ludwig.
\newblock Covariance matrices and valuations.
\newblock {\em Adv. in Appl. Math.}, 51(3):359--366, 2013.

\bibitem{MaRealvaluedvaluations2016}
D.~Ma.
\newblock Real-valued valuations on {S}obolev spaces.
\newblock {\em Sci. China Math.}, 59(5):921--934, 2016.

\bibitem{McMullenValuationsEulertype1977}
P.~McMullen.
\newblock Valuations and {E}uler-type relations on certain classes of convex
  polytopes.
\newblock {\em Proc. London Math. Soc. (3)}, 35(1):113--135, 1977.

\bibitem{MussnigVolumepolarvolume2019}
F.~Mussnig.
\newblock Volume, polar volume and {E}uler characteristic for convex functions.
\newblock {\em Adv. Math.}, 344:340--373, 2019.

\bibitem{Ober$L_p$Minkowskivaluations2014}
M.~Ober.
\newblock {$L_p$}-{M}inkowski valuations on {$L^q$}-spaces.
\newblock {\em J. Math. Anal. Appl.}, 414(1):68--87, 2014.

\bibitem{ParkKinematicformulasreal2002}
H.~Park.
\newblock {\em Kinematic formulas for the real subspaces of complex space forms
  of dimension 2 and 3}.
\newblock PhD thesis, University of Georgia, 2002.

\bibitem{TasakiGeneralizationKahlerangle2001}
H.~Tasaki.
\newblock Generalization of {K}\"{a}hler angle and integral geometry in complex
  projective spaces.
\newblock In {\em Steps in differential geometry ({D}ebrecen, 2000)}, pages
  349--361. Inst. Math. Inform., Debrecen, 2001.

\bibitem{TradaceteVillanuevaValuationsBanachLattices2020}
P.~Tradacete and I.~Villanueva.
\newblock Valuations on {B}anach lattices.
\newblock {\em Int. Math. Res. Not. IMRN}, (1):287--319, 2020.

\bibitem{TrudingerWangHessianmeasures.I1997}
N.~S. Trudinger and X.-J. Wang.
\newblock Hessian measures. {I}.
\newblock {\em Topol. Methods Nonlinear Anal.}, 10(2):225--239, 1997.

\bibitem{TrudingerWangHessianmeasures.II1999}
N.~S. Trudinger and X.-J. Wang.
\newblock Hessian measures. {II}.
\newblock {\em Ann. of Math. (2)}, 150(2):579--604, 1999.

\bibitem{TsangValuations$Lp$spaces2010}
A.~Tsang.
\newblock Valuations on {$L^p$}-spaces.
\newblock {\em Int. Math. Res. Not. IMRN}, (20):3993--4023, 2010.

\bibitem{TsangMinkowskivaluations$Lp$2012}
A.~Tsang.
\newblock Minkowski valuations on {$L^p$}-spaces.
\newblock {\em Trans. Amer. Math. Soc.}, 364(12):6159--6186, 2012.

\bibitem{TsengYauCohomologyHodgetheory2012}
L.-S. Tseng and S.-T. Yau.
\newblock Cohomology and {H}odge theory on symplectic manifolds: {I}.
\newblock {\em J. Differential Geom.}, 91(3):383--416, 2012.

\bibitem{WangSemiValuations$BVRn$2014}
T.~Wang.
\newblock Semi-valuations on {${\mathrm BV}(\Bbb R^n)$}.
\newblock {\em Indiana Univ. Math. J.}, 63(5):1447--1465, 2014.

\bibitem{WannererIntegralgeometryunitary2014}
T.~Wannerer.
\newblock Integral geometry of unitary area measures.
\newblock {\em Adv. Math.}, 263:1--44, 2014.

\bibitem{Wannerermoduleunitarilyinvariant2014}
T.~Wannerer.
\newblock The module of unitarily invariant area measures.
\newblock {\em J. Differential Geom.}, 96(1):141--182, 2014.

\end{thebibliography}

\Addresses

\end{document}